\documentclass[11pt,a4paper,reqno,final]{amsart}
\usepackage{amssymb,amsmath,graphics,verbatim}
\usepackage{latexsym}
\usepackage{eucal}
\usepackage{a4wide}
\usepackage[colorlinks]{hyperref}
\usepackage[normalem]{ulem}
\usepackage{soul}
\usepackage{cancel}
\usepackage[textwidth=100pt]{todonotes}

\usepackage{marginnote}
\newcommand{\stkout}[1]{\ifmmode\text{\sout{\ensuremath{#1}}}\else\sout{#1}\fi}

\newcommand{\FLO}{{\rm FLO}}

\newcommand{\p}{\partial}

\newcommand{\D}{\mathbf D}

\newcommand \B{\mathcal B}

\newcommand\R{\rr}

\newcommand{\rr}{\mathbb{R}}
\newcommand{\nn}{\mathbb{N}}

\newcommand{\singsupp}{{\rm singsupp}}

\newcommand{\oo}{{\mathcal O}}
\newcommand{\T}{{\bf T}}
\newtheorem{theorem}{Theorem}
\newtheorem{lemma}[theorem]{Lemma}
\newtheorem{proposition}[theorem]{Proposition}
\newtheorem{corollary}[theorem]{Corollary}

\theoremstyle{definition}
\newtheorem{definition}[theorem]{Definition}

\numberwithin{equation}{section}
\numberwithin{theorem}{section}

\DeclareMathOperator \spn {span}

\DeclareMathOperator \WF {WF}

\DeclareMathOperator \supp {supp}

\usepackage[asymmetric, hmarginratio=2:3]{geometry}

\usepackage[breakall, fit]{truncate}
\usepackage{accsupp}
\usepackage[notref, notcite]{showkeys}

\usepackage{mathtools}
\DeclarePairedDelimiter\pair{\langle}{\rangle}

\title[Single Observer Determination of Lorentzian Manifold]{Determining Lorentzian manifold from non-linear wave observation at a single point}

\author[]{Medet Nursultanov}
\address {Department of Mathematics and Statistics, University of Helsinki, Helsinki, Finland}
\email{medet.nursultanov@gmail.com}

\author[]{Lauri Oksanen}
\address {Department of Mathematics and Statistics, University of Helsinki, Helsinki, Finland}
\email{lauri.oksanen@helsinki.fi}

\author[]{Leo Tzou}
\address {Korteweg-de Vries Institute for Mathematics, University of Amsterdam, Amsterdam, Netherlands}
\email{leo.tzou@gmail.com}

\subjclass[2000]{53C20, 35R30, 35L70}
\keywords{wave front propagation, non-linear wave interaction, geodesics, Lorentzian geometry, inverse problems}
\thanks{Corresponding author: M.Nursultanov; E-mail address: medet.nursultanov@gmail.com.}

\begin{document}
\begin{abstract}
We consider an inverse problem for a non-linear hyperbolic equation. We show that conformal structure of a Lorentzian manifold can be determined by the source-to-solution map evaluated along a single timelike curve. We use the microlocal analysis of non-linear wave interaction.

\end{abstract}

\maketitle

\section{Introduction}
Since the introduction of linearization methods for recovering the background geometry from data about solutions of non-linear hyperbolic equations \cite{Kurylev:2018aa}, many works have followed \cite{ultra21,Chen2019,Chen2020,Hoop2019,Hoop2019a,Feizmohammadi2019,Hintz2021,Hintz2020,Kurylev2014,Lassas2017,Tzou2021,Uhlmann2020,Uhlmann2019,UhlmannZhang2022JDE,UhlmannSIAManal2023,Zhang2023}. We also mention works studying inverse problems for non-linear hyperbolic equations \cite{FuYao,Kian,NakamuraVashisthWatanabe2021,Romanov2023,Barreto2020InvProbIm,BarretoStefanov2022,Barreto2020,WangZhou}. For an overview of the recent progress, see \cite{LassasRio2018,UhlmannZhai2021}. In most of these cases the data acquisition geometry roughly consists of sources and measurements taken in a small, albeit open, space-time tube around a timelike curve. There has also been work where the measurements set and source set are disjoint \cite{feizmohammadi2020inverse}, though in these cases the measurements are also taken in open tubes. 

In this work we propose a model which requires less measurements to be taken in comparison to the previous results mentioned above. In particular, we will show that, to recover the background geometric structure, one only needs to measure the solution of a non-linear wave along a single timelike curve. We emphasize, however, that we still need to arrange the sources in an open tubular neighbourhood of this time-like curve.

We now give a precise formulation of our inverse problem. Let $(M,g)$ be a globally hyperbolic, $(1+3)$-dimensional Lorentzian manifold, where the metric $g$ is of signature $(-,+,+,+)$. Global hyperbolicity allows us to write 
\begin{align*}
	&M = \R \times M_0;\\
	&g = \beta(t,x')(- dt^2 + \kappa(t,x')),
\end{align*}
where $(t,x') = (x_0,x_1,x_2,x_3)$ are local coordinates on $M$, $\beta$ is a smooth positive function on $M$, and $(M_0,\kappa(t,\cdot))$ is a Rimannian manifold with the metric depending smoothly on $t\in \mathbb R$.
 For $p\in M$, we denote by $J^+(p)$ and $J^-(p)$ to be the causal past and future, respectively. We consider semilinear wave equation
 \begin{equation}\label{cubic wave}
 	\begin{cases}
 		\Box_g u+ u^3 = f, & \text{on } M,\\
 		u\mid_{ (-\infty, 0)\times M_0}= 0,
 	\end{cases}
 \end{equation}
 where 
 \begin{equation*}
 	\Box_g := (- \det g)^{-\frac{1}{2}} \partial_j ((- \det g)^{\frac{1}{2}}g^{jk}\partial_k)
 \end{equation*}
 is the wave operator on $(M,g)$. We aim to extract geometric information by measuring waves on a single point which is represented as a smooth future pointing timelike curve
 \begin{equation*}
 	\mu :[-1,1] \to M.
 \end{equation*}
 Let $(s^-, s^+)\subset\subset (-1,1)$ and let $\Omega\subset\subset (0,\infty) \times M_0$ be any open set containing $\mu([s^-,s^+])$. For all $f\in C^N_c(\Omega)$, define the single observer source-to-solution map by
$$Lf := \mu^* u_f,$$ 
where $u_f$ is the unique solution to \eqref{cubic wave}. Hence, $L$ represents the measurements of waves produced by sources supported on $\Omega$ and observed at $\mu$. Our main result is the following:

\begin{theorem}
\label{main theorem}
 The operator $f\mapsto Lf$ determines the topological, differential, and conformal structure of $\D:= J^+(\mu(s^-)) \cap J^{-}(\mu(s^+))$.
\end{theorem}

We now provide a brief outline of our strategy. For $\epsilon = (\epsilon_0, \cdots, \epsilon_6)$, we take
\begin{equation*}
	f_{\epsilon} = \sum_{j=0}^6 \epsilon_j f_j
\end{equation*}
to be source and $u$ to be corresponding solution to \eqref{cubic wave}. Then it is easy to see that the first order linearization, $u_{j} := \partial_{\epsilon_j} u\mid_{\epsilon = 0}$, will solve the linear wave equation with source $f_j$. It is also easy to see that $u_{123}$ is also the solution of the linear wave equation but now with source $u_1u_2u_3$. It was first observed by \cite{Kurylev:2018aa} that, due to propagation of singularity, having products of the previous linearization acting as the source allows us to recover data about broken lightrays leaving then returning to the tubular neighbourhood $\Omega$.  This technique is often called multiple linearization.

Our situation is more challenging as it only provides us information about return rays along a single timelike curve, which on its own is not sufficient to deduce the background geometry. To overcome this, we introduce even higher order linearizations $u_{123jk}$ where $j,k \in\{0,4,5,6\}$ so that $u_{123}u_j u_k$ are now source terms. If we choose $f_j$ for $j\in \{0,4,5,6\}$ appropriately so that $u_0$, $u_4$, $u_5$, and $u_6$ are supported in the right places, we can make it so that information always get propagated back to the timelike curve $\mu([s_-,s_+])$. From this we obtain geometric information which is encoded in the Three-to-One Scattering Relation introduced in \cite{feizmohammadi2020inverse}. It was shown in \cite{feizmohammadi2020inverse} that this information uniquely determines the background geometry up to a conformal factor.

A special case of Theorem \ref{main theorem} for the case when the Lorentzian metric is ultrastatic was done in \cite{Tzou2021} where the geometric structure of a Riemannian manifold was recovered via measurement at a single point. In \cite{LassasLiimatainenPotenciano-MachadoTyni2022}, it was noted that for the hyperbolic case, it might be enough to measure the Dirichlet-to-Neumann map integrated against a suitable fixed function. Similar type of single point inverse problem using non-linearity in the elliptic setting was considered in \cite{salo2023inverse}. Let us also mention that linearization technique was first used in the context of elliptic inverse problem in \cite{FeizmohammadiOksanen2020} and \cite{LassasLiimatainenLinSalo2021}.

All the aforementioned works employ non-linearity as a fundamental instrument, whereas the corresponding problems for linear cases remain unsolved. This is due to the lack of uniqueness results for the linear case. Currently, established uniqueness results for linear hyperbolic equations with vanishing initial data are based on Tataru’s unique continuation theorem \cite{Tataru1995,Tataru1999}. Consequently, these results require the coefficients to be constant or real-analytic in the time variable.  We mention works \cite{BelishevKurylev1992,Helin_Lassas_Oksanen_Saksala2018,LassasNursultanovOksanenYlinen2023,LassasOksanen2010,LassasUhlmann2001} where the background geometry was recovered form solution data of linear wave equation.

\section{Notation}
Let $(M,g)$ be a globally hyperbolic, $1+3$ dimensional Lorentzian manifold. Global hyperbolicity allows us to write $M = \R \times M_0$ and $g = \beta(t,x')(- dt^2 + \kappa(t,x'))$ for a family of Riemannian metrics $\kappa(t,\cdot)$ on $M_0$. Endow $M$ with a Riemannian metric $G$ and we use the same letter to denote the Sasaki metric on $T^*M$.

Assuming that $(M,g)$ is time-oriented enables us to establish the direction of time and define time-like and causal paths that point towards the future and the past. We recall that a smooth path $\mu:(a,b) \rightarrow M$ is timelike if $g(\dot{\mu}, \dot{\mu}) < 0$ on $(a,b)$. We say that $\mu$ is causal if $g(\dot{\mu}, \dot{\mu}) \leq 0$ and $\dot{\mu} \neq 0$. For $p$, $q\in M$, $p\ll q$ means that they are distinct and there is a future pointing timelike path from $p$ to $q$. Similarly, $p < q$ means that they are distinct and there is a future pointing causal path from $p$ to $q$. We say that $p\leq q$ if $p=q$ or $p<q$. The chronological future of $p\in M$ is the set
$$
    I^+(p) := \{q\in M: p\ll q\}
$$
and causal future of $p$ is the set
$$
    J^+(p) := \{q\in M: p\leq q\}.
$$
Analogically, we set chronological past, $I^-(p)$, and causal past, $J^-(p)$. For a set $A\subset M$, we denote 
$$
    J^{\pm}(A) := \bigcup_{p\in M} J^{\pm}(p).
$$
For $W\subset M$, let
\begin{equation*}
    L^{*,+}W:=\bigcup_{p\in W} L_p^{*,+}M \subset T^*W
\end{equation*}
the bundle of future pointing lightlike covectors. Analogically, we define $L^{*,-}W$. The projection from the cotangent bundle $T^*M$ to the base point
of a vector is denoted by $\pi: T^*M \rightarrow M$.

For $p$, $q\in M$, we define the separation function as
\begin{equation*}
    \tau(p,q):= 
    \begin{cases}
        \sup_{\alpha} \int_0^1 \sqrt{g(\dot{\alpha},\dot{\alpha})} & \text{if } p<q,\\
        0 & \text{otherwise},
    \end{cases}
\end{equation*}
where the supremum is taken over all piecewise smooth causal paths $\alpha:[0,1]\rightarrow M$ from $p$ to $q$. 

For $x=(t,x')$ and $\eta \in L_x^{*,+}M$, let $s(\eta)\in (0,\infty]$ be the maximal value  for which geodesic $\gamma_\eta: [0,s(\eta)) \mapsto M$ is defined. We define the cut function by 
\begin{equation*}
	\rho(\eta) := \sup \{ s\in [0, s(\eta)):\; \tau(x, \gamma_s(\eta)) = 0\}.
\end{equation*}
We define $\rho(\eta)$ also for $\eta \in L_x^{*,-}M$ by the above expression but with respect to the opposite time orientation. 

\section{A Three-to-One Scattering Relation}

We will prove Theorem \ref{main theorem} by using the notion of a three-to-one scattering relation defined in \cite{feizmohammadi2020inverse}. Before providing the definition, we introduce necessary notations. 
Let $H$ denote the Hamiltonian vector field and $\Sigma(\Box_g)$ denote the characteristic set associated with $\sigma[\Box_g]$, that is 
\begin{equation*}
	H(x,\xi) := 2g^{ij}\xi_j\partial_{x^i} - (\partial_{x^i}g^{jk}\xi_j\xi_k)\partial_{\xi_i},
\end{equation*}
\begin{equation*}
	\Sigma(\Box_g) := \{(x,\eta)\in T^*M\setminus 0:  \; (\eta,\eta)_g = 0\}.
\end{equation*}
We denote by $\Phi_s$ the flow of $H$. For the set $K\subset \Sigma(\Box_g)$, we define the future flowout by
\begin{equation*}
	\FLO^+(K) = \left\{ (y, \eta)\in \Sigma(\Box_g): \; (y,\eta) = \Phi_s(x,\xi), \; s\in \mathbb{R}, (x,\xi)\in K, y \geq x \right\}.
\end{equation*}
Past flowout $\FLO^-(K)$ is defined analogically.

Now we are ready to recall the definition of the three-to-one scattering relation
introduced in \cite{feizmohammadi2020inverse}:

\begin{definition} 
	Let $\Omega \subset M$ be open and nonempty. 
	A relation $\mathcal R \subset  (L^{*,+} \Omega)^4$ is a three-to-one scattering relation if the following two conditions hold:
	\begin{itemize}
		\item[(R1)] If $(\xi_0, \xi_1, \xi_2, \xi_3) \in \mathcal R$
		then 
		\begin{align*}
			\pi\circ\FLO^-(\xi_0) \cap \bigcap_{j=1}^3 \pi\circ \FLO^+(\xi_j) \neq\emptyset.
		\end{align*}
            \item[(R2)] The set $\mathcal{R}$ contains all $(\xi_0, \xi_1, \xi_2, \xi_3) \in  (L^{*,+} \Omega)^4$ which satisfy 
            \begin{itemize}
			\item[(a)] The bicharacteristics through $\xi_j$ are distinct, that is, $\FLO(\xi_j) \ne \FLO(\xi_k)$ when $j \ne k$.
			\item[(b)] There are $y \in M$, $s_0 \in (-\rho(\xi_0), 0)$,
			$s_j \in (0, \rho(\xi_j))$, $j=1,2,3$, such that 
			$y = \gamma_{\xi_j}(s_j)$ for $j=0,1,2,3$.
			\item[(c)] Writing $\eta_j$ for the covector version of 
			$\dot \gamma_{\xi_j}(s_j)$, it holds that $\eta_0 \in \spn(\eta_1, \eta_2, \eta_3)$.
		\end{itemize}
	\end{itemize} 
\end{definition} 
Relation $(\mathrm{R1})$ means that if $(\xi_0,\xi_1,\xi_2,\xi_4)\in \mathcal{R}$ it is necessary for the future pointing geodesics of $\xi_1$, $\xi_2$, $\xi_3$, and the past pointing geodesic of $v_0$, to intersect at some point $y$. While, relation $(\mathrm{R2})$ means that for $(\xi_0,\xi_1,\xi_2,\xi_4)\in \mathcal{R}$ it is sufficient if the geodesics of $\xi_1$, $\xi_2$, $\xi_3$, and $\xi_0$ intersect at some point $y$ before their cut points and the velocity at $y$ of the geodesic corresponding to $\xi_0$ belongs to the span of velocities at $y$ corresponding to $\xi_1$, $\xi_2$, and $\xi_3$; see Figure \ref{fig:enter-label}.

\begin{figure}
    \centering
    \includegraphics[width=0.5\textwidth]{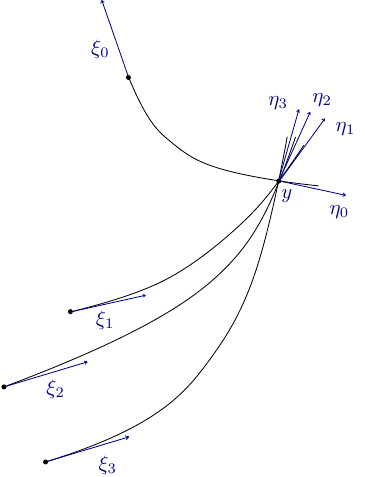}
    \caption{}
    \label{fig:enter-label}
\end{figure}

Note that since we can always consider smaller $\Omega$, we will assume without loss of generality that 
\begin{eqnarray*}
	\Omega \subset\subset  J^-(\mu(1))\backslash J^-(\mu(-1)).
\end{eqnarray*}
Therefore, for every $x = (t, x') \in \Omega$, there exists a unique $\xi^{\mu}(x) \in L^+_{x} M/\R^+$ such that $\gamma_{\xi^{\mu}(x)}([0, \rho(\xi^\mu(x)))) \cap \mu([-1,1])$ is non-empty. Furthermore, by definition of $\rho(\cdot)$ this intersection contains exactly one element which we will denote by 
\begin{eqnarray}
	\label{xhat(x0)}
	\hat x (x) := \gamma_{\xi^{\mu}(x)}([0, \rho(\xi^\mu(x)))) \cap \mu([-1,1]).
\end{eqnarray}

Denote by $\T_{x'}\subset M$ to be the codimension $3$ submanifold $\{(t,x')\mid t\in (-\infty,\infty)\}$. Clearly, for any lightlike covector $\xi\in L^{*,+}_{x_0}M$, we have that 
$$ T^*_{x}M = \spn{\xi} \oplus N^*_{x} \T_{x'}.$$
Therefore, we have the following
\begin{lemma}\label{lem_ximu}
Let $\xi\in L^{*,+}\Omega$ and $x=\pi(\xi)$, then there exists a unique $\xi'\in N^*_{x}\T_{x'}/\R^+$ and $c\in \mathbb{R}$ such that  
\begin{equation*}
	\xi^\mu(x) = c\xi + \xi'.
\end{equation*}
If $c\xi = \xi^\mu(x)$ then $\xi'$ here is understood to be the zero covector. 
\end{lemma}
The last lemma allows us to define the map
\begin{align}\label{def_nu}
    & \nu: L^{*,+}_{x}\Omega/\R^+ \to N^*_{x}\T_{x'}/\R^+,\\
    \nonumber & \nu (\xi) := \xi'.
\end{align}

Given an ordered quadruple of covectors $(\xi_0, \xi_1, \xi_2, \xi_3) \in( L^{*,+}\Omega)^4$ with $\pi(\xi_0) \in\bigcap_{j=1}^3 J^+(\pi(\xi_j))$, we construct the following distributions. In what follows, if $\xi\in L^{*,+}M$ and $h>0$ we denote
\begin{align}\label{def_B}
    \B_h(\xi) := \{ \lambda \eta \in T^*M  \mid d_G(\eta, \xi) <h, \lambda \in \R^+\}. 
\end{align}
Moreover, from now on we will use notation $x_j: = \pi(\xi_j)$ for $j=0, 1,2,3.$

\subsection{Distributions $f_1$, $f_2$, and $f_3$}
To begin, we create sources $f$ that are compactly supported near a specific point in $M$ and have a wavefront set that is microlocalized near a single direction.

For  $j=1,2,3$ and $h>0$, let $\omega_j(x,\xi;h)\in \Psi^0(M)$ with a homogeneous of degree zero symbol $\omega_j(x,\xi;h)$ satisfying
\begin{equation*}
    \omega_j(x_j,\pm\xi_j;h) =1, \quad \text{and} \quad \supp(\omega_j(\cdot; \pm\xi_j, h)) \subset  \B_h(\xi_j)\cup \B_h(-\xi_j).
\end{equation*}
Define the compactly supported conormal distributions $f_j\in I(M,    \B_h(\xi_j)\cup \B_h(-\xi_j))$
\begin{eqnarray}
	\label{fj}
	f_j(\cdot; \xi_j, h) :=  \omega_j(x,D;h)\langle D\rangle^{-N}\delta_{x_j}
\end{eqnarray}
where $\delta_{x_j}$ is the Dirac delta function and $\langle D\rangle$ is an elliptic classical psedodifferential operator of order $1$. Here, $N>0$ is large enough so that $f_j\in C_c^{N'}$ is smooth as we need. Then the following result holds
\begin{lemma}
\label{linear wave}
Suppose for each $h>0$, $u_j(\cdot ; h)$ is the solution to
\begin{equation*}
    \begin{cases}
        \Box_g u_j(\cdot;\xi_j, h) = f_j(\cdot;\xi_j, h),\\
        u_j\arrowvert_{(-\infty, 0)\times M_0} = 0
    \end{cases}
\end{equation*}
then 
$$u_j\in I(T^*_{x_j}M, \FLO^+((\B_h(\xi_j)\cup \B_h(-\xi_j))\cap L_{x_j}^{*,+} M)).$$
Furthermore, $\sigma[u_j](\pm\xi) \neq 0$ for all $\xi\in \FLO^+(\xi_j)$.
\end{lemma}

\subsection{Distributions $f_0$, $f_4$, $f_5$, and $f_6$}

For each fixed $h>0$ let $a>0$ be a parameter so that $0<a<h$. We let $\chi_a(\cdot; x_0)\in C^\infty_c(B_G(x_0;a))$ satisfy $\chi_a(x_0  ;x_0) = 1$. Define 
\begin{eqnarray}
\label{f0f4}
f_0(x; x_0,a) = f_4(x; x_0,a) = \Box_g \chi_a(x;x_0)
\end{eqnarray}
 so that $\chi_a(\cdot;x_0)$ is the unique solution to
 \begin{equation*}
     \begin{cases}
         \Box_g u (\cdot; x_0,a) = f_0(\cdot; x_0,a) = f_4(\cdot; x_0,a),\\
         u\arrowvert_{(-\infty, 0)\times M_0} = 0.
     \end{cases}
 \end{equation*}

For $\tilde x = (\tilde t , \tilde x') \in \pi\circ \FLO^+(\xi_0)\cap \Omega,$
let $\chi_a(\cdot; \tilde x)\in C^\infty_c(B_G(\tilde x;a))$ satisfy $\chi_a(\tilde     x;\tilde x) = 1$. Define $f_6(\cdot ;\tilde x, a)$ by
\begin{eqnarray}
\label{f6}
f_6(\cdot ;\tilde x, a) := \Box_g\left(\chi_a  \langle D\rangle^{-N} \delta_{\T_{\tilde x'}}\right)
\end{eqnarray}
where $\delta_{\T_{\tilde x'}}$ the distribution given by integrating along $\T_{\tilde x'}$. Observe that 
\begin{eqnarray}
\label{u6}
u_6(\cdot ;\tilde x, a) = \left(\chi_a \langle D\rangle^{-N} \delta_{\T_{\tilde x'}}\right)
\end{eqnarray}
is the unique solution to 
\begin{equation*}
    \begin{cases}
        \Box_gu_6(\cdot ;\tilde x, a) = f_6(\cdot ;\tilde x, a),\\
        u_6\arrowvert_{(-\infty, 0)\times M_0} = 0.
    \end{cases}
\end{equation*}
Finally, we set
\begin{eqnarray}
\label{f5}
f_5(\cdot ;\tilde x, a) := \Box_g \chi_a(\cdot; \tilde x)
\end{eqnarray}
so that $u_5 (\cdot ;\tilde x, a) = \chi_a(\cdot; \tilde x)$ is the unique solution to the wave equation with source $f_5$. 
\subsection{Seventh-order interaction of waves for the non-linear wave equation}\label{linearisation}
Let $(\xi_0,\xi_1,\xi_2, \xi_3) \in (L^{*,+}\Omega)^4$ and
$$\tilde x = (\tilde t , \tilde x') \in \pi\circ \FLO^+(\xi_0)\cap \Omega.$$ 
For $h>0$ and $a\in (0,h)$, define sources $\{f_j\}_{j=0}^6$ as in the previous section. Next, we introduce a vector of seven variables denoted by $\epsilon = (\epsilon_0,\cdots,\epsilon_6)$. For the non-linear wave equation \eqref{cubic wave}, we denote by 
$$u_\epsilon = u_\epsilon(\cdot; \xi_0,\xi_1,\xi_2, \xi_3, \tilde x, a, h)$$
its solution when the source is given by
$$f = \sum_{j=0}^6 \epsilon_j f_j.$$
We also set
    \begin{align}\label{linearized_sol}
    &\nonumber u_j := \left.\partial_{\epsilon_j} u_\epsilon\right|_{\epsilon = 0}   &  j\in \{0,\cdots,6\}; \\
    &\nonumber u_{jk} := \left.\partial_{\epsilon_j}\partial_{\epsilon_k} u_\epsilon\right|_{\epsilon = 0}       &   j,k\in \{0,\cdots,6\};\\
    & \cdots\\
    &\nonumber u_{0123456} := \left.\partial_{\epsilon_0}\dots \partial_{\epsilon_6} u_\epsilon\right|_{\epsilon = 0}.
\end{align}
Then it follows that
\begin{equation}\label{box_uj}
    \begin{cases}
        \Box_g u_j = f_j,\\
        u_j\arrowvert_{(-\infty, 0)\times M_0} = 0,
    \end{cases}
\end{equation}
so that functions $\{u_j\}_{j = 1}^6$ coincide with the functions described in the previous section. In particular,
\begin{equation*}
    u_0 = u_4 = \chi_a(\cdot ; x_0), \qquad u_5 = \chi_a(\cdot; \tilde{x}), \qquad u_6 = \chi_a(\cdot; \tilde{x}) \langle D\rangle^{-N} \delta_{\T_{\tilde x'}}.
\end{equation*}
Similarly, one can check that $u_{jk} = 0$ for $j$, $k\in \{0,\cdots,6\}$ distinct. Moreover, for $j$, $k$, $l$, $\alpha$, $\beta\in \{0,\cdots,6\}$ distinct,
\begin{equation}\label{box u123}
    \begin{cases}
        \Box_g u_{jkl} = -6u_ju_ku_l,\\
        u_{jkl}\arrowvert_{(-\infty, 0)\times M_0} = 0,
    \end{cases}
\end{equation}
and
\begin{equation}\label{box u01234}
    \begin{cases}
        \Box_g u_{jkl\alpha\beta} = -\sum\limits_{\sigma\in S_5}u_{\sigma(j)} u_{\sigma(k)} u_{\sigma(l)\sigma(\alpha) \sigma(\beta)},\\
        u_{jkl\alpha\beta}\arrowvert_{(-\infty, 0)\times M_0} = 0,
    \end{cases}
\end{equation}
Here, $S_n$ is the permutation group on $\{j,k,l,\alpha,\beta\}$, that is a set of bijective operators from $\{j,k,l,\alpha,\beta\}$ to itself.

We will also need the  7-fold linearization $u_{0123456}$:
$$\Box_g u_{0123456} = \sum\limits_{\sigma\in S_7}\left(u_{\sigma(5)} u_{\sigma(6)} u_{\sigma(0)\sigma(1) \sigma(2)\sigma(3) \sigma(4)}+ u_{\sigma(0)} u_{\sigma(1)\sigma(2)\sigma(3)}u_{\sigma(4)\sigma(5)\sigma(6)}\right).$$

\subsection{A Three-to-One Scattering Relation}
Let $(\xi_0,\xi_1,\xi_2, \xi_3) \subset (L^{*,+}\Omega)^4$ based at points $(x_0,x_1,x_2,x_3)$, respectively, and $u_\epsilon$ be the solution of \eqref{cubic wave} described in the previous section. Let us consider the following two conditions for $(\xi_0,\xi_1,\xi_2, \xi_3)$:

\begin{enumerate}
\item\label{no return} We say that $(\xi_0, \xi_1, \xi_2,\xi_3)$ satisfies the {\em non-return} condition if 
\begin{eqnarray}
\label{no intersect mu}
\hat x(x_0) \notin \pi\circ \FLO^+(\xi_0) 
\end{eqnarray}
and 
\begin{eqnarray}
\label{distinction}
x_0 \notin \bigcup_{j=1}^3 \left(\pi\circ\FLO^+(\xi_j) \right).
\end{eqnarray}

\item \label{desirable} We say that $(\xi_0, \xi_1,\xi_2,\xi_3)$ satisfies the {\em desirable condition} if for any open set $\mathcal{O} \subset \Omega$ containing $x_0$ there exists $\tilde x\in {\mathcal O} \cap  \pi\circ\FLO^+(\xi_0)$ satisfying
\begin{equation*}
    \tilde x \neq x_0, \qquad \tilde x\notin\mu([0,1]),
\end{equation*}
\begin{eqnarray}
\label{hat x tilde x}
\hat x(\tilde x) \notin \bigcup_{j=1}^3 \left(\pi\circ\FLO^+(\xi_j) \right),
\end{eqnarray}
and that $u_\epsilon = u_\epsilon (\cdot; \xi_0, \xi_1,\xi_2, \xi_3, \tilde x, a, h)$ satisfies  
\begin{eqnarray}
\label{observed singularity}
\mu^{-1}(\hat x(\tilde x))\in \singsupp(\mu^*u_{0123456}),
\end{eqnarray}
for any sufficiently small $h>0$ and $0<a<h$.
\end{enumerate}

We now set
\begin{eqnarray}
\label{scattering relation}
R := \left\{(\xi_0, \xi_1, \xi_2,\xi_3)\mid x_0 \in \bigcap_{j=1}^3 I^+(x_j), \mbox{conditions \eqref{no return} and \eqref{desirable} are satisfied}\right\}
\end{eqnarray}

Observe that apriori this set is not uniquely determined by the source-to-solution operator $f\mapsto Lf$. This is because we do not apriori know whether condition \eqref{no return} is satisfied. We take care of this by

\begin{proposition}\label{boomerang}
    Let $\xi_0$, $\xi_1\in L^{*,+}\Omega$ such that $x_0 \in I^+(x_1)$, then the map $f\mapsto Lf$ determines whether the relation
    \begin{equation}\label{xi_0_belongs_FLO_xi_1}
        x_0 \in \pi\circ \FLO^+(\xi_1)
    \end{equation}
    is satisfied.
\end{proposition}
In order to prove it we need the following modification of Lemma 2.12 in \cite{Tzou2021}:
\begin{lemma}\label{lemma_2.12}
    Let $S$ be hypersurface containing $x^*$. Suppose $u\in I(M, N^*S)$ is a conormal distribution and $\alpha$ be a curve intersecting $S$ transversally at $x^*$. If $\sigma[u](x^*,\xi)\neq 0$ for $\xi\in N_{x^*}^*S$, then the distribution $\alpha^*u\in D'(\mathbb R)$ is singular at $0$.
\end{lemma}
\begin{proof}
    Since the intersection of $\alpha$ and $S$ is transversal, we can find the local coordinates system $z=(z_0,z_1,z_2,z_3) \in \mathbb R^{1 + 3}$ such that $S = \{z_0 = 0\}$ and
    \begin{equation*}
        \alpha(t) = (t, \alpha_1(t), \alpha_2(t), \alpha_3(t)).
    \end{equation*}
    In these coordinates, we write 
    \begin{equation*}
        u(z) = \int_{\mathbb R} e^{i\theta z_0} a(z, \theta) d\theta,
    \end{equation*}
    where the symbol $a(z, \theta)$ satisfies 
    \begin{equation}\label{non_zero_sym_conc}
        |a(x^*,\theta)| \geq C |\theta|^N
    \end{equation}
    for some $N\in \mathbb N$, $C>0$, and all $\theta>1$. Then
    \begin{equation*}
        \alpha^*u(t) \int_{\mathbb R} e^{i\theta t} \tilde a(t, \theta) d\theta,
    \end{equation*}
    where
    \begin{equation*}
        \tilde a(t, \theta) = a(\alpha(t), \theta).
    \end{equation*}
    Then, the stationary phase method for the Fourier transform of $\alpha^*u(t)$ gives
    \begin{multline*}
        \mathcal{F} [\chi\alpha^*u] (\xi) = \int_{\mathbb R} \int_{\mathbb R} e^{it\eta} \chi(t) \tilde a(t, \eta + \xi) dt d\eta = \chi(0) \tilde a(0,\xi) + R_1[\chi\alpha^*u](\xi)\\
        =\chi(0) a(x^*,\xi) + R_1[\chi\alpha^*u](\xi).
    \end{multline*}
    where $\chi \in C_0^\infty (\mathbb R)$ is an arbitrary function supported in a neighbourhood of zero and $R_1[\chi\alpha^*u]$ is the reminder containing all lower order terms. Therefore, due to \eqref{non_zero_sym_conc}, it follows that $\alpha^*u$ is singular at $0$.
\end{proof}

\begin{proof}[Proof of Proposition \ref{boomerang}]
    We examine two cases: one where $x_0$ lies on the curve $\mu((-1,1))$ and another where it does not. 
    
    \textbf{Case 1.} Assume that $x_0\in \mu((-1,1))$, that is there exists $s_0\in (-1,1)$ such that $x_0 = \mu(s_0)$. For $h>0$ and
    $$\tilde \xi_1 \in \FLO^+(\xi_1) \cap T^*\Omega,$$
    let $f_1 = f_1(\cdot; \tilde{\xi}_1, h)$ be the function defined by \eqref{fj} with $\tilde \xi_1$ in place of $\xi_1$ and $u_1(\cdot; \tilde{\xi}_1, h)$ be the corresponding solution of \eqref{box_uj}.
    
    We will show that condition \eqref{xi_0_belongs_FLO_xi_1} is satisfied if and only if there exists 
    $$\tilde \xi_1 \in \FLO^+(\xi_1) \cap T^*\Omega$$
    such that 
    \begin{equation}\label{sing_at_mus0}
        s_0 \in \singsupp(\mu^* u_1(\cdot; \tilde \xi_1, h))
    \end{equation}
    for all $h>0$. This equivalence implies that the map $f\mapsto Lf$ determines relation \eqref{xi_0_belongs_FLO_xi_1}.
    
    Suppose that $\mu(s_0)\in \pi\circ\FLO^+(\xi_1)$. We choose $\tilde\xi_1$ being small perturbation of $\xi_1$ along $\pi\circ\FLO^+(\xi_1)$ such that $\pi(\tilde\xi_1)$ and $\mu(s_0)$ are not conjugate to each other along $\pi\circ\FLO^+(\xi_1)$. Then, we know that in an open neighbourhood of $\mu(s_0)$,
    $$\FLO^+\left(\left(\B_h(\tilde \xi_1)\cup \B_h(-\tilde\xi_1)\right)\cap L^*_{\pi\left(\tilde \xi_1\right)}M\right)$$
    is the conormal bundle $N^*S$ of some lightlike hypersurface $S$. By Proposition 6.6 in \cite{Melrose:1979aa}, we know that 
    \begin{equation*}
        u_1(\cdot; \tilde \xi_1,h) \in I\left(\FLO^+\left(\left(\B_h(\tilde\xi_1)\cup \B_h(-\tilde\xi_1) \right)\cap L^*_{\pi\left(\tilde \xi_1\right)}M\right)\right)
    \end{equation*}
    near $\mu(s_0)$. Since $\mu(s_0)$ is not conjugate to $\pi(\tilde \xi_1)$ along $\FLO^+(\xi_1)$, the distribution $u_1(\cdot;\tilde\xi_1,h)\in I(N^*S)$ in a small neighbourhood of $\mu(s_0)$. Furthermore, if $\hat \xi_0 \in N^*_{\mu(s_0)} S$, then since $$\sigma[f_1(\cdot;\tilde \xi_1, h)](\tilde \xi_1) \neq 0,$$
    we have that 
    $$\sigma[u(\cdot; \tilde \xi_1, h)](\hat \xi_0) \neq 0;$$
    see relation (6.7) in \cite{Melrose:1979aa}. By Lemma \ref{lemma_2.12}, we conclude that $\mu^* u_1(\cdot;\tilde\xi_1,h)$ is singular at $s_0$. 

    Now suppose the contrary that $\mu(s_0)\notin \pi\circ \FLO^+(\xi_1)$. Assume that there exists
    $$\tilde \xi_1 \in \FLO^+(\xi_1) \cap T^*\Omega$$
    such that \eqref{sing_at_mus0} is fulfilled. Then, for sufficiently small $h>0$, it follows that 
    $$\mu(x_0)\notin \pi\circ \FLO^+\left(\left(\B_h(\tilde \xi_1)\cup \B_h(-\tilde\xi_1)\right)\cap L^*_{\pi\left(\tilde \xi_1\right)}M\right).$$
    By Theorem 23.2.9 in \cite{Hormander:2007aa}, this means that $u_1(\cdot;\tilde \xi_1,h)$ solving \eqref{box_uj} is smooth at $\mu(s_0)$. This contradicts the assertion \eqref{sing_at_mus0}.

    \textbf{Case 2.} Now we assume that $x_0\notin \mu((-1,1))$. We will show that condition \eqref{xi_0_belongs_FLO_xi_1} is satisfied if and only if there exists $$\tilde \xi_1 \in \FLO^+(\xi_1) \cap T^*\Omega$$ such that for any open neighbourhood $\mathcal{O}$ of $x_0$ and sufficiently small $h>0$ we can find $\tilde{x}\in \mathcal{O}$ so that 
    \begin{equation}\label{sing_at_xhat}
        \mu^{-1}(\hat{x}(\tilde{x})) \in  \singsupp \left(\mu^* u_{156}(\cdot; \tilde\xi_1, \tilde x, a, h)\right),
    \end{equation}
    for all $a\in (0,h)$. Recall that $u_{156}$ is the solution of \eqref{box u123}, that is,
    \begin{equation}\label{156 return}
        \begin{cases}
            \Box u_{156} = -6u_1 \chi_a \left(\chi_a  \langle D\rangle^{-N} \delta_{\T_{\tilde x'}}\right),\\
            u_{156}\mid_{(-\infty,0)\times M_0} = 0
        \end{cases}
    \end{equation}
    where the distribution $u_1 = u_1(\cdot; \tilde \xi_1, h)$ is defined as in Lemma \ref{linear wave} with $\tilde\xi_1$ in place of $\xi_1$ and $\chi_a(\cdot; \tilde x)\in C^\infty_c(B_G(\tilde x;a))$ with $\chi_a(\tilde     x;\tilde x) = 1$.

    Suppose that \eqref{xi_0_belongs_FLO_xi_1} holds. As in the first case, we choose $\tilde\xi_1$ to be a small perturbation of $\xi_1$ along $\FLO^+(\xi_1)$ such that $\tilde x_1 = \pi(\tilde\xi_1)$ and $x_0$ are not conjugate to each other along $\pi\circ\FLO^+(\xi_1)$. By Proposition 6.6 in \cite{Melrose:1979aa}, for all sufficiently small $h>0$ and sufficiently small neighbourhood $\mathcal O$ of $x_0$, the distribution $u_1(\cdot; \tilde \xi_1,h)$ belongs to $I(N^*S)$ in $\mathcal O$ for some lightlike hypersurface $S$. Find $\tilde x\in S$ so that $\tilde \xi\in N^*_{\tilde x}S$ satisfies 
    \begin{eqnarray}
    \label{not directly to mu}
    \hat x(\tilde x) \notin\pi\circ \FLO^+(\tilde \xi).
    \end{eqnarray}
    At the point $\tilde x$ there is a unique covector $\hat \xi\in L^{*,+}_{\tilde x}M$ such that $\hat x(\tilde x) \in \pi\circ\FLO^+(\hat \xi)$. Due to \eqref{not directly to mu}, we have that $\hat\xi \neq \tilde \xi.$ We decompose the solution of \eqref{156 return} as 
    $$u_{156} = u_{\rm sing} + u_{\rm reg},$$ 
    where $u_{\rm sing}$ and $u_{\rm reg}$ are the solution of the following equations
    \begin{equation}\label{ureg return}
    \begin{cases}
        \Box u_{\rm reg} = -6\left(1-\omega(x,D) \right)\left(u_1 \chi_a \left(\chi_a  \langle D\rangle^{-N} \delta_{\T_{\tilde x'}}\right)\right),\\
        u_{\rm reg}\mid_{(-\infty,0)\times M_0} = 0
    \end{cases}
\end{equation}
and
\begin{equation}\label{using return}
    \begin{cases}
        \Box u_{\rm sing} = -6\omega(x,D)\left(u_1 \chi_a \left(\chi_a  \langle D\rangle^{-N} \delta_{\T_{\tilde x'}}\right)\right),\\
        u_{\rm sing}\mid_{(-\infty,0)\times M_0} = 0.
    \end{cases}
\end{equation}
The microlocal cutoff $\omega(x,D)\in \Psi^0(M)$ is homogeneous of degree zero, supported in $\B_{\hat h}(\hat \xi)\cup \B_{\hat h}(-\hat\xi)$, and takes the value $1$ in $\B_{\hat h/2}(\hat \xi)\cup \B_{\hat h/2}(-\hat\xi)$. Since $\hat\xi \neq \tilde \xi$, we have that if $\hat h>0$ is sufficiently small,
\begin{eqnarray}
\label{tilde xi not in}
\tilde\xi\notin \B_{\hat h}(\hat \xi)\cup \B_{\hat h}(-\hat\xi).
\end{eqnarray}

We need to show that for all $a>0$ sufficiently small,
\begin{eqnarray}
\label{in using not in ureg}
\hat x(\tilde x) \notin \singsupp(u_{\rm reg}) \quad \text{and} \quad \mu^{-1}(\hat x(\tilde x)) \in \singsupp(\mu^*u_{\rm sing}).
\end{eqnarray}

We first look at the inhomogeneous term in \eqref{using return}. Since $u_1\in I(N^*S)$ near $\tilde x$ with $\sigma[u_1](\tilde x, \tilde\xi)\neq 0$, \eqref{tilde xi not in} indicates that 
$$\omega(x,D)\left(u_1 \chi_a \left(\chi_a  \langle D\rangle^{-N} \delta_{\T_{\tilde x'}}\right)\right) \in I(T^*_{\tilde x}M)$$ 
with nonvanishing symbol at $\hat \xi$. By (6.7) in \cite{Melrose:1979aa}, near $\hat x(\tilde x)$, $u_{\rm sing}$ is Langrangian distribution whose symbol along $\FLO^+(\hat \xi)$ is nonvanishing. And since the lightlike geodesic segment from $\tilde x$ to $\hat x(\tilde x)$ can be extended without creating conjugate points, $u_{\rm sing}$ is actually a conormal distribution near $\hat x(\tilde x)$. Using Lemma \ref{lemma_2.12}, we get that $\mu^{-1}(\hat x(\tilde x)) \in \singsupp(\mu^* u_{\rm sing})$. This is the second part of \eqref{in using not in ureg}

For the first part of \eqref{in using not in ureg}, we observe that $\pi\circ \FLO^+(\hat \xi)$ is the unique causal curve joining $\tilde x$ and $\hat x(\tilde x)$. Therefore, we can choose $a>0$ sufficiently small such that if $\xi\in L^*_{\hat x(\tilde x)}M$ satisfies
$$\pi\circ\FLO^-(\xi) \cap T^*B_G(\tilde x;a)\neq \emptyset,$$
then 
$$\FLO^-(\xi)\cap T^*B_G(\tilde x;a) \subset \B_{\hat h/2}(\hat \xi)\cup \B_{\hat h/2}(-\hat\xi).$$
But $\omega(x,D)$ is constructed so that the inhomogeneous term of \eqref{ureg return} is  not microlocally supported in $\B_{\hat h/2}(\hat \xi)\cup \B_{\hat h/2}(-\hat\xi)$. So by Theorem 23.2.9 in \cite{Hormander:2007aa}, $\hat x(\tilde x)\notin \singsupp(u_{\rm reg})$.

Finally, assume there exists $\tilde \xi_1 \in \FLO^+(\xi_1) \cap T^*\Omega$ such that \eqref{sing_at_xhat} holds. Then, for sufficiently small $h>0$ and all open sets $\mathcal O$ containing $x_0$ there is $\tilde x\in \mathcal O$ such that $u_{156}(\cdot; \tilde \xi_1, \tilde x, h,a)$ solving \eqref{156 return} has a singularity at $\hat x(\tilde x)$ for all $a\in (0,h)$. Since $x_0\notin \mu((-1,1))$ we limit the statement to those open sets $\mathcal O$ which do not intersect $\mu((-1,1))$. Therefore, $\tilde x\notin \mu((-1,1))$. Hence, we conclude that $u_1(\cdot; \tilde \xi_1,h)$ has a singularity singularity at some $\tilde x\in \mathcal O$ for all open sets $\mathcal O$ containing $x_0$. This means that for all sufficiently small $h>0$, $x_0\in \singsupp(u_1(\cdot; \tilde \xi_1,h))$. Therefore, 
$$x_0\in \pi\circ \FLO^+(\B_h(\xi_1) \cap L^*_{x_0}M)$$
for all $h>0$, and hence, \eqref{xi_0_belongs_FLO_xi_1} holds. 
\end{proof}


The next Proposition verifies that $\overline R$ defined in \eqref{scattering relation} is a three-to-one scattering relation as per \cite{feizmohammadi2020inverse}:
\begin{proposition}
\label{R is scattering relation}
The relation $\overline R$ defined in \eqref{scattering relation} is a three-to-one scattering relation.
\end{proposition}
Theorem \ref{main theorem} is then a corollary of Proposition \ref{R is scattering relation} by using the result of \cite{feizmohammadi2020inverse}.

\section{Flowout from interaction curve}
In this section, for a spacelike curve $K\subset M$, we will identify $(N^*K \cap L^{*,+}M)/\R^+$ with
$$\{\eta\in N_y^*K \cap L_y^{*,+}M\mid y = (t,y') \in K,\ \eta = -dt + \eta',\ \eta'\in T^*_{y'}M_0,\ \|\eta'\|_{\kappa(t,\cdot)} = 1\}.$$
Let $\hat y \in K$ and suppose
\begin{equation*}
    \hat\eta\in (L^{*,+}_{\hat y}M\cap N^*_{\hat y} K)/\R^+   \qquad and \qquad \hat s\in (0,\rho(\hat \eta)).
\end{equation*}
We begin with two auxiliary results.
\begin{lemma}
\label{neighbourhood extremal}
There exists $\delta>0$ and an open subset
$$U\subset (N^*K \cap L^{*,+}M)/\R^+$$
containing $\hat\eta$ such that $\hat s+ \delta < \rho(\eta)$ for all $\eta\in U$.
\end{lemma}
\begin{proof}
Suppose the contrary. Then there exist sequences $\delta_j>$ and
$$\eta_j\in (L^{*,+}_{\hat y}M\cap N^*_{\hat y} K)/\R^+$$
such that $\delta_j \to 0$, $\eta_j \to \hat{\eta}$, and
$$\delta_j + \hat s \geq \rho(\eta_j).$$
Taking the limit we get
$$\hat s \geq \liminf \rho(\eta_j) \geq \rho(\hat\eta)$$
by lower semi-continuity of $\rho(\cdot)$. This contradicts our assumption that $\hat s\in (0,\rho(\hat\eta))$.\end{proof}

\begin{lemma}\label{Sj is lightlike}
    Let 
    $$\Lambda := \FLO^+(N^*K \cap L^{*,+}M),$$
    then for any $\lambda\in \Lambda$ there exists $V\in T^*M$ an open conic neighbourhood of $\lambda$ such that $V\cap \Lambda$ is a conic subset of $N^*S$ for some lightlike submanifold $S\subset M$ of codimension $1$.
\end{lemma}
\begin{proof}
    Since $\Lambda$ is a Lagrangian manifold Proposition 3.7.2 in \cite{Duistermaat:1996aa} implies that any $\lambda\in \Lambda$ has conic neighbourhood $V\in T^*M$ such that $S:=\pi(V\cap \Lambda)$ is a $k$-dimensional submanifold of $M$ and $V\cap \Lambda$ is an open subset of $N^*S$.

    Next, we show that $k=1$. Assume that $k>0$. Let $x\in S$. Since $\Lambda\cap V$ is an open conic subset of $N^*S$, we conclude that $\pi^{-1}(x) \cap (V\cap \Lambda)$ is an open subset of $N^*S$, therefore, contains an open convex subset $U$. Since $k>1$, $U$ contains two linearly independent lightlike covectors together with their nontrivial convex combinations which are not lightlike. This contradicts to the fact that all covectors in $U\subset \Lambda$ are lightlike, so that $k=1$. 
\end{proof}

\begin{proposition}
\label{flowout is conormal}
There exist sequences
\begin{equation*}
    \eta_j \in (N_{y_k}^*K \cap L_{y_k}^{*,+}M)/\R^+  \qquad and \qquad s_j \in (0,\rho(\eta_j))
\end{equation*}
converging to $\hat \eta$ and $\hat s$, respectively, such that the following condition holds: For each $j\in \mathbb N$, we can find sufficiently small $h>0$ and open set $\oo_j\subset M$ containing $z_j:= \gamma_{\eta_j}(s_j)$ so that 
$$\FLO^+(N^*\Gamma_{j;h}\cap L^{*} M)\cap T^*\oo_{j} $$
is the conormal bundle of a lightlike hypersurface, where
\begin{equation*}
    \Gamma_{j;h} := \{y\in K\mid d_G(y,y_j)<h\}.
\end{equation*}

Furthermore, for each $j\in \nn$ there exists a conic open set $U_{j} \subset T^*M$ containing $\eta_j$ such that 
$$\pi\circ \FLO^+(N^*\Gamma_{j;h}\cap L^*M)\cap\oo_{j} \cong \left(U_j\cap N^*\Gamma_{j;h} \cap L^{*,+}M\right) /\R^+ \times (-\delta + s_j, s_j +\delta)$$
for some $\delta>0$ depending on $j\in \nn$ and $h>0$.
\end{proposition}
\begin{proof}
Let $\lambda_j \in \Lambda$ be a sequence converging to $\hat \lambda := \Phi_{\hat s} (\hat\eta)$. By Lemma \ref{Sj is lightlike}, for each $j\in \mathbb N$, there exists an open conic set $V_j\subset T^*M$ containing $\lambda_j$ and lightlike hypersurfaces $S_j\subset M$ of codimension one such that $V_j\cap \Lambda$ is a conic open subset of $N^*S_j$. So the projection $\pi(V_j\cap \Lambda) = \oo_j \cap S_j$ for some open set $\oo_j\subset M$. We set $z_j := \pi(\lambda_j) \in S_j$ and choose sequences
\begin{equation*}
    \eta_j\in (N^*K \cap L^{*,+}M)/\R^+  \qquad and \qquad s_j >0
\end{equation*}
converging to $\hat \eta$ and $\hat s$, respectively, such that $\Phi_{s_j} (\eta_j) = \lambda_j$. Due to Lemma \ref{neighbourhood extremal}, for sufficiently large $j\in \mathbb N$, it follows that 
\begin{eqnarray}
\label{sj before cut point}
s_j < \rho(\eta_j).
\end{eqnarray}
Without loss of generality, we assume that this holds for all $j\in \mathbb N$. Therefore, by applying Lemma \ref{neighbourhood extremal} to $(\eta_j, s_j)$, we find $\delta_j>0$ and open subset 
$$U_j\subset (N^*K \cap L^{*,+}M)/\R^+$$
containing $\eta_j$ such that $s_j + \delta_j< \rho(\eta)$ for all $\eta\in U_j$. Without loss of generality, we may choose $U_j$ and $\delta_j$ to be small so that 
\begin{eqnarray}
\label{Phis in conormal}
\Phi_{s}(\eta) \in V_j \cap \Lambda = N^*S_j\cap T^*\oo_j,
\end{eqnarray} 
for all $s\in (s_j -\delta_j, s_j + \delta_j)$ and $\eta\in U_j$. Consequently,
\begin{equation}\label{on the hypersurface}
    \pi\circ\Phi_s(\eta) \in S_j\cap \oo_j,
\end{equation}
for all $s\in (s_j -\delta_j, s_j + \delta_j)$ and $\eta\in U_j$.


We now need to verify that for each $j\in \nn$ we can choose $h>0$ sufficiently small and $\tilde \oo_j\subset\subset \oo_j$ containing $z_j$ so that the entire flowout from $N^*\Gamma_{j;h} \cap L^{*,+}M$ satisfies
$$\FLO^+(N^*\Gamma_{j;h} \cap L^{*,+}M) \cap T^*\tilde \oo_{j} \subset N^*S_{j}.$$
Suppose this statement fails to hold. Then, for each $j\in\nn$, there is a sequence
$$\{\eta_{j;k}\}_{k\in \nn} \subset (N^*K \cap L^{*,+}M)/\R^+$$ 
with $\pi(\eta_{j;k})$ converging to $y_j$ and some bounded sequence $r_{j;k} \in \R$ such that
$$\pi\circ\Phi_{r_{j;k}}(\eta_{j;k}) \to z_j$$
but $\Phi_{r_{j;k}}(\eta_{j;k}) \notin N^*S_j$ for all $k\in \nn$.

Suppose that $\eta_{j;k}\in U_j$ for some $k\in\nn$. Due to \eqref{Phis in conormal}, this implies that
$$r_{j;k}\notin (s_j - \delta_j, s_j + \delta_j).$$
If for fixed $j\in \nn$, there are infinitely many $\eta_{j;k}\in U_j$, then there exists a limit point $(\tilde \eta_j, \tilde r_j)$ of the sequence $( \eta_{j;k},  r_{j;k})$ such that
\begin{equation*}
    \tilde \eta_j\in N^*_{y_j}K \cap L^{*,+}_{y_j}M\cap \bar U_j, 
    \qquad 
    \tilde r_j \notin  (s_j-\delta_j, s_j + \delta_j),
\end{equation*}
and $\pi\circ \Phi_{\tilde r_j}(\hat\eta_j) = z_j$. Clearly $\tilde \eta_j \neq \eta_j$ or we will have a self-intersecting lightlike geodesic. So we now have two distinct causal paths joining $y_j$ and $z_j$ contradicting \eqref{sj before cut point}; see Lemma 6.5 in \cite{feizmohammadi2020inverse}. Therefore, we can conclude that for each fixed $j\in \nn$ there are only finitely many $\eta_{j;k}$ in $U_{j}$.

Without loss of generality, $\eta_{j;k} \notin U_j$ for all $k\in \nn$. Then, there exists a limit point $(\tilde \eta_j, \tilde r_j)$ of the sequence $( \eta_{j;k},  r_{j;k})$ such that $\tilde \eta_j \notin U_j$. On the other hand,
$$\pi\circ \Phi_{r_{j;k}}(\eta_{j;k}) \to z_j = \pi\circ \Phi_{s_j}(\eta_j),$$
and hence, we have two distinct causal curves joining $z_j$ and $y_j$, contradicting \eqref{sj before cut point}.
\end{proof}

\section{On the regularity of the interaction of waves}
In this section, we investigate the regularity of the waves defined in Section \ref{linearisation}.

\begin{lemma}
\label{uj smooth}
Assume that \eqref{distinction} and \eqref{hat x tilde x} are satisfied. We have that for $j=1,2,3$ and $h>0$ sufficiently small, $u_j\mid_{\mathcal O}$ is smooth and that $\hat x(\tilde x) \notin\singsupp(u_j)$.  
\end{lemma}
\begin{proof} By condition \eqref{distinction}, for all $h>0$ sufficiently small, we have that 
$$\pi\circ\FLO^+(\B_h(\xi_j) \cap L_{x_j}^{*,+}M)\cap\mathcal O = \emptyset.$$
By Lemma \ref{linear wave}, each $u_j$ satisfies that 
$$\singsupp(u_j ) \subset \pi\circ \FLO^+(\B_h(\xi_j) \cap L_{x_j}^{*,+}M).$$
Therefore, $\singsupp(u_j) \cap \mathcal O = \emptyset$. 

Similarly, by condition \eqref{hat x tilde x}, we have that if $h>0$ is small enough, $$\hat x(\tilde x) \notin \bigcup_{j=1}^3 \pi\circ\FLO^+(\B_h(\xi_j) \cap L_{x_j}^{*,+}M).$$
So $\hat x(\tilde x) \notin\singsupp(u_j)$.

\end{proof}

\begin{lemma}\label{uabc}
    For $a\in (0,h)$ sufficiently small, it follows that $u_{\alpha\beta\gamma} = 0$, where $\alpha$, $\beta$, $\gamma\in \{0,4,5,6\}$ are distinct.
\end{lemma}
\begin{proof}
    Since $\alpha$, $\beta$, $\gamma$ are distinct, we may assume that $\alpha<\beta<\gamma$. In particular, $\alpha\in \{0,4\}$ and $\gamma\in \{5,6\}$. Then $\supp(u_\alpha)$ and $\supp(u_\gamma)$ do not intersect, and hence, \eqref{box u123} becomes $\Box_g u_{\alpha\beta\gamma} = 0$. Taking into account the trivial initial condition, we conclude $u_{\alpha\beta\gamma} = 0$.
\end{proof}

\begin{lemma}\label{ujab}
    Assume that \eqref{distinction} and \eqref{hat x tilde x} are satisfied. Then, for $a\in (0,h)$ sufficiently small, it follows:
    \begin{enumerate}
        \item For $j\in \{1,2,3\}$ and $\alpha$, $\beta\in \{0,4,5\}$ distinct,
        $$\WF(u_{j\alpha\beta}) = \emptyset;$$
        \item For $j\in \{1,2,3\}$ and $\alpha \in \{0,4,5\}$,
        $$\WF(u_{j\alpha 6})\subset \WF(u_6);$$ 
        \item For $j\in \{1,2,3\}$, $\alpha \in \{0,4\}$, and $\beta\in \{5,6\}$,
        $$ u_{j\alpha\beta} = 0.$$
    \end{enumerate}
\end{lemma}

\begin{proof}
    $(1)$ In Section \ref{linearisation}, we established that 
    \begin{equation*}
        \Box_g u_{j\alpha\beta} = -6 u_j u_\alpha u_\beta.
    \end{equation*}
    Recall that $\tilde{x}$, $x_0 \in \mathcal{O}$ and 
    \begin{equation*}
        u_0 = u_4 = \chi_a(\cdot; x_0) \qquad u_5 = \chi_a(\cdot; \tilde{x}).
    \end{equation*}
    Therefore, for sufficiently small $a\in (0,h)$, $u_\alpha$ and $u_\beta$ are supported in $\mathcal{O}$. Hence, Lemma \ref{uj smooth} implies that $u_{j}u_{\alpha}u_{\beta}\in C^\infty(M)$, and consequently, $u_{j\alpha\beta}\in C^\infty(M)$.

    $(2)$ Similarly, we know that 
    \begin{equation*}
        \Box_g u_{j\alpha6} = -6 u_j u_\alpha u_6.
    \end{equation*}
    For sufficiently small $a\in (0,h)$, we have that $\supp({u_\alpha})\subset\mathcal{O}$ and $\supp({u_6})\subset \mathcal{O}$ compactly. By Lemma \ref{uj smooth}, $\left.u_j\right|_{\mathcal{O}} \in C^\infty (\mathcal{O})$, so that $\WF(u_j u_\alpha u_6) \subset \WF(u_6)$. Since $\WF(u_6)$ is spacelike, by ellipticity of $\Box_g$ in spacelike directions, we obtain $\WF(u_{j\alpha 6})\subset \WF(u_6)$. 

    $(3)$ By definition $\supp(u_\alpha)\subset B_a(x_0)$ and $\supp(u_\beta)\subset B_a(\tilde{x})$. Since $x_0$ and $\tilde{x}$ are distinct, for sufficiently small $a\in (0,h)$, it follows that 
    \begin{equation*}
        \Box_g u_{j\alpha\beta} = -6 u_ju_\alpha u_\beta = 0.
    \end{equation*}
    Taking into account the trivial initial condition, we obtain $u_{j\alpha\beta} = 0$.
\end{proof}

\begin{lemma}\label{ujka}
    Assume that \eqref{distinction} and \eqref{hat x tilde x} are satisfied. Then, for $a\in (0,h)$ sufficiently small,
    \begin{enumerate}
        \item For $j$, $k\in \{1,2,3\}$ distinct and $\alpha\in \{0,4,5\}$,
        $$\WF(u_{jk\alpha}) = \emptyset;$$
        \item For $j$, $k\in \{1,2,3\}$ distinct,
        $$\WF(u_{jk 6})\subset \WF(u_6).$$ 
    \end{enumerate}
\end{lemma}

\begin{proof}
    $(1)$ By Lemma \ref{uj smooth}, $u_ju_k$ is smooth on $\mathcal{O}$. For sufficiently small $a>0$, we obtain $\supp(u_\alpha)\subset \mathcal{O}$, so that $u_ju_ku_\alpha \in C_0^\infty(M)$. Hence, since $u_{jk\alpha}$ satisfies 
    \begin{equation*}
        \Box_g u_{jk\alpha} = -6u_ju_ku_\alpha
    \end{equation*}
    and the trivial initial condition, we obtain $u_{jk\alpha} \in C^\infty(M)$.

    $(2)$ Similarly, by construction, $u_6$ is supported in $\mathcal{O}$ with spacelike wavefront set. Since $u_{jk6}$ satisfies
    \begin{equation*}
        \Box_g u_{jk6} = -6u_ju_ku_6
    \end{equation*}
    with the trivial initial condition, the elliptic regularity gives $\WF(u_{jk 6})\subset \WF(u_6)$.
\end{proof}

\begin{lemma}
    Let $\alpha \in \{0,4\}$, $\beta\in\{5,6\}$, and $m$, $n \in \{0,\cdots,6\}$ be distinct numbers. Then, for $a\in (0,h)$ sufficiently small,
    \begin{equation*}
        u_\alpha u_{mn\beta} = 0.
    \end{equation*}
\end{lemma}
\begin{proof}
    By the finite speed of the wave propagation, $\supp(u_{mn\beta})$ is a subset of the future causal cone of $B_G(\tilde x; a)$, and hence, does not intersect $\supp(u_\alpha)$.
\end{proof}

\begin{lemma}\label{five_WF}
    Assume that \eqref{distinction} and \eqref{hat x tilde x} are satisfied. Then, for $a\in (0,h)$ sufficiently small,
    \begin{enumerate}
        \item For $j$, $k\in \{1,2,3\}$ distinct,
        $$\WF(u_{jk045}) = \emptyset;$$
        \item  For $j$, $k\in \{1,2,3\}$ distinct and $\alpha$, $\beta\in \{0,4,5\}$ distinct,
        $$\WF(u_{jk\alpha\beta6})\subset \WF(u_6);$$
        \item For $j\in \{1,2,3\}$,
        $$\WF(u_{j0456}) \subset \WF(u_6);$$
    \end{enumerate}
\end{lemma}

\begin{proof}
    $(1)$ If $a\in (0,h)$ is sufficiently small, then the supports of $u_0$ and $u_4$ do not intersect the support of $u_5$, so that
    \begin{equation*}
        u_0u_5 u_{jk4} = u_4u_5 u_{jk0} = 0.
    \end{equation*}
    Moreover, by Lemma \eqref{ujab}, 
    \begin{equation*}
        u_ju_k u_{045} = u_ju_0 u_{k45} =u_ju_4 u_{k05} = u_ku_0 u_{j45} = u_ku_4 u_{j05} = 0.
    \end{equation*}
    Therefore, \eqref{box u01234} gives
    \begin{equation*}
        \Box_g u_{jk045} = -2 ( u_{j}u_5 u_{k04} + u_{k}u_5 u_{j04} + u_0u_4u_{jk5}). 
    \end{equation*}
    Note that $u_{5}$ is supported in $\mathcal{O}$ and $u_j$, $u_k$ are smooth in $\mathcal{O}$, moreover, $u_0$, $u_4 \in C^\infty(M)$. Therefore, Lemmas \ref{ujab} and \ref{ujka} imply that the right-hand side of the last equation is smooth so that $(1)$ holds.

    $(2)$ Without lost of generality, we assume that $\alpha < \beta$ so that $\alpha\in \{0,4\}$. As is the previous part, for sufficiently small $a\in (0,h)$, it follows that
    \begin{equation*}
        u_\alpha u_6u_{jk\beta} = u_j u_\beta u_{k\alpha 6} = u_k u_\beta u_{j\alpha 6} = u_j u_k u_{\alpha \beta 6} = 0.
    \end{equation*}
    For the last term, we used Lemma \ref{uabc}. Therefore, \eqref{box u01234} becomes
    \begin{equation*}
        \Box_g u_{jk\alpha\beta 6} = -( u_j u_\alpha u_{k\beta 6} + u_j u_6 u_{k\alpha \beta} + u_k u_\alpha u_{j\beta 6} + u_k u_6 u_{j\alpha\beta} + u_\alpha u_\beta u_{jk 6} + u_\beta u_6 u_{jk\alpha})
    \end{equation*}
    As in the previous part, $u_j u_\alpha$, $u_k u_\alpha$, $u_\alpha u_\beta$, $u_{j\alpha\beta}$, and $u_{jk\alpha}$ are smooth. By Lemmas \ref{ujab} and \ref{ujka}, $\WF(u_{k\beta 6})$, $\WF(u_{jk 6})\subset \WF(u_6)$. Therefore, the wavefront set of the right-hand side belongs to $\WF(u_6)$, and hence, the ellipticity of $\Box_g$ in spacelike directions gives $(2)$.

    $(3)$ By Lemmas \ref{uabc} and \ref{ujab}, \eqref{box u01234} becomes 
    \begin{equation*}
        \Box_g u_{j0456} = -2(u_5 u_6u_{j04}  + u_0u_4u_{j56}).
    \end{equation*}
    By Lemma \ref{ujab}, $u_{j04}$ is smooth and $\WF(u_{j56})\subset \WF(u_6)$. Hence, $\WF(u_{j0456}) \subset \WF(u_6)$. Here, we also used that $u_0$, $u_4$, and $u_5$ are smooth.
\end{proof}

\begin{lemma}\label{five_zeros}
    Let $\alpha \in \{0,4\}$, $\beta\in\{5,6\}$, and $m$, $n$, $\tau$, $\gamma \in \{0,\cdots,6\}$ be distinct numbers. Then, for $a\in (0,h)$ sufficiently small and , it follows
    \begin{equation*}
        u_\alpha u_{mn\tau\gamma\beta} = 0.
    \end{equation*}
\end{lemma}
\begin{proof}
    We recall that 
    \begin{equation*}
        \Box_g u_{mn\tau\gamma\beta} = - \sum_{\sigma\in S_4}  u_{\sigma(m)}u_\beta u_{\sigma(n)\sigma(\tau)\sigma(\gamma)}
        - \sum_{\sigma\in S_4}  u_{\sigma(m)}u_{\sigma(n)}u_{\sigma(\tau)\sigma(\gamma)\beta}
    \end{equation*}
    where $S_4$ is the permutation group on $\{m,n,\tau,\gamma\}$. The first sum is supported in $B_G(\tilde x;a)$, while the second sum is supported in the future causal cone of $B_G(\tilde x;a)$. Therefore, $\supp(u_{mn\tau\gamma\beta})$ is a subset of the future causal cone of $B_G(\tilde x;a)$. Therefore, for sufficiently small $a>0$, $\supp(u_{mn\tau\gamma\beta})$ does not intersect $\supp(u_\alpha)$. 
\end{proof}

We set $u_{\rm reg}$ to be the solution of 
\begin{eqnarray}
\label{ureg}
\Box_g u_{\rm reg} = \sum_{\sigma\in S_{7},\sigma\notin S_5} u_{\sigma(5)}u_{\sigma(6)} u_{\sigma(0)\sigma(1)\sigma(2)\sigma(3) \sigma(4)} +  \sum_{\sigma\in S_{7}} u_{\sigma(0)}u_{\sigma(1)\sigma(2)\sigma(3)}u_{\sigma(4)\sigma(5)\sigma(6)}
\end{eqnarray}
where we use $S_5\subset S_7$ to denote the subset of $S_7$ which maps $\{0,\dots, 4\}$  to itself. Then we can write
\begin{equation}\label{decomp_u0123456}
    u_{0123456} = u_{\rm reg} + u_{\rm sing},
\end{equation}
where
\begin{equation}\label{using}
    \begin{cases}
        \Box_g u_{\rm sing} = C u_{5} u_6 u_{01234},\\
        u_{\rm sing}\arrowvert_{(-\infty, 0)\times M_0} = 0,
    \end{cases}
\end{equation}
for some $C\in \mathbb{N}$.

\begin{lemma}\label{ureg_is_smooth}
    Assume that \eqref{distinction} and \eqref{hat x tilde x} are satisfied. Then, for sufficiently small $a\in(0,h)$, $u_{\rm reg}$ is smooth at $\hat{x}(\tilde x)$.
\end{lemma}
\begin{proof}
    Due to Lemma, we obtain
    \begin{multline*}
        \sum_{\sigma\in S_{7},\sigma\notin S_5} u_{\sigma(5)}u_{\sigma(6)} u_{\sigma(0)\sigma(1)\sigma(2)\sigma(3) \sigma(4)}  =  A\sum_{\sigma\in S_3} \left( u_{\sigma(j)}u_{\sigma(k)}u_{\sigma(l)0456}\right.\\
        \left.+ u_{\sigma(j)}u_{5}u_{\sigma(k)\sigma(l)046}
         + u_{\sigma(j)}u_{6}u_{\sigma(k)\sigma(l)045}\right),
    \end{multline*}
    for some $A\in \mathbb{N}$, where $S_3$ is the permutation group on $\{j,k,l\}$. Lemma \ref{five_WF} implies that the wavefront set of the right-hand side belongs to 
    \begin{equation*}
        \WF(u_6)\cup \WF(u_1) \cup \WF(u_2) \cup \WF(u_3) \cup \left(\bigcup_{j,k}^3 (\WF(u_j) + \WF(u_k)) \right).
    \end{equation*}
    Similarly, Lemma \ref{five_zeros} implies
    \begin{equation*}
        \sum_{\sigma\in S_{7}} u_{\sigma(0)}u_{\sigma(1)\sigma(2)\sigma(3)}u_{\sigma(4)\sigma(5)\sigma(6)} = B\sum_{\sigma\in S_3} u_{\sigma(j)} u_{\sigma(k)04} u_{\sigma(l)56}.
    \end{equation*}
    Due to Lemma \ref{ujab}, the wavefront set of the right-hand side belongs to
    \begin{equation*}
        \WF(u_6)\cup \WF(u_1)\cup \WF(u_2)\cup \WF(u_3).
    \end{equation*}
    We summarize,
    \begin{equation*}
        \WF (\Box_g u_{\rm reg}) \subset \WF(u_6)\cup \WF(u_1) \cup \WF(u_2) \cup \WF(u_3) \cup \left(\bigcup_{j,k}^3 (\WF(u_j) + \WF(u_k)) \right),
    \end{equation*}
    and hence,
    \begin{equation*}
        \WF (\Box_g u_{\rm reg}) \cap L^*M \subset \WF(u_1) \cup \WF(u_2) \cup \WF(u_3),
    \end{equation*}
    As the wavefront set of $u_1$, $u_2$, and $u_3$ are lightlike flowouts from point sources, we conclude that $u'_{\rm reg}$ solving \eqref{ureg} must satisfy
    $$\WF(u_{\rm reg}) \subset \WF(u_1)\cup \WF(u_2) \cup \WF(u_3).$$
The above inclusion in conjunction with Lemma \ref{uj smooth} completes the proof.
\end{proof}

We will also need more information about the singularity of $u_{01234}$. We set $v_{\rm reg}$ to be the solution of 
\begin{equation*}
\begin{cases}
    \Box_g v_{\rm reg} = \sum_{\sigma \in S_5\backslash S_3} u_{\sigma(0)} u_{\sigma(4)} u_{\sigma(1)\sigma(2)\sigma(3)}\\
    v_{\rm reg}\arrowvert_{(-\infty, 0)\times M_0} = 0,
\end{cases}
\end{equation*}
where $S_3\subset S_5$ is the subset of $S_5$ which maps the set $\{1,2,3\}$ to itself. Then we can write
\begin{equation}\label{decomp_u01234}
    u_{01234} = v_{\rm reg} + v_{\rm sing}
\end{equation}
where
\begin{equation*}
\begin{cases}
    \Box_g v_{\rm sing} = C u_0u_4u_{01234},\\
    v_{\rm sing}\arrowvert_{(-\infty, 0)\times M_0} = 0,
\end{cases}
\end{equation*}
for some $C\in \mathbb N$. Next, we show that $v_{\rm reg}$ is smooth near $x_0$.

\begin{lemma}
\label{vreg is smooth}

Assume that \eqref{distinction} and \eqref{hat x tilde x} are satisfied. Then there exists an open neighbourhood $\mathcal O$ of $x_0$ such that, for all $0<a<h$ sufficiently small 
$$\singsupp(v_{\rm reg}) \cap {\mathcal O}= \emptyset$$.
\end{lemma}

\begin{proof} 
The source terms on the right side are of the form $u_{jk\alpha}u_\beta u_l$ and $u_{j04} u_k u_l$
where $j,k,l\in\{1,2,3\}$ and $\alpha,\beta \in \{0,4\}$ are distinct. Using Lemma \ref{ujab} and Lemma \ref{ujka} we see that
$$\WF\left(\sum_{\sigma \in S_5\backslash S_3} u_{\sigma(0)} u_{\sigma(4)} u_{\sigma(1)\sigma(2)\sigma(3)}\right) \subset \bigcup_{j=1}^3 \WF(u_j)\cup \bigcup_{j,k=1}^3 \left(\WF(u_j) + \WF(u_k)\right).$$
By Lemma \ref{uj smooth} we have that there exists a small neighbourhood $\mathcal O$ containing $x_0$ such that 
\begin{eqnarray}
\label{no wf in T*O}
\WF(u_j)\cap T^*{\mathcal O} = \emptyset,\ j = 1,2,3.
\end{eqnarray}
Then
$$\WF\left(\sum_{\sigma \in S_5\backslash S_3} u_{\sigma(0)} u_{\sigma(4)} u_{\sigma(1)\sigma(2)\sigma(3)}\right)  \cap T^*{\mathcal O} = \emptyset.$$

So for this choice of $\mathcal O$, if $\xi\in T^*\mathcal O$ is in the wavefront of  $v_{\rm reg}$ then it must be in $L^*\mathcal O$ and $\FLO^{-}(\xi)$ intersects $\bigcup_{j=1}^3 \WF(u_j)$. However, as each $\WF(u_j)$ is the lightlike flowout of a point source, if $\FLO^-(\xi)$ intersects one of $\WF(u_j)$ then $\xi\in \WF(u_j)$ contradicting \eqref{no wf in T*O}.\end{proof}

\section{Necessary Condition to Belonging to $R$}

The point of this section is to prove that $(\xi_0,\xi_1,\xi_2,\xi_3)\in R$ implies
\begin{equation*}
    \pi\circ\FLO^-(\xi_0) \cap \bigcap_{j=1}^3 \pi\circ \FLO^+(\xi_j) \neq\emptyset.
\end{equation*}
Throughout this section, it is assumed that $(\xi_0,\xi_1,\xi_2,\xi_3)\in R$, unless explicitly stated otherwise. Moreover, we choose an open set $\mathcal{O}\subset\Omega$ containing $x_0$ small enough so that
\begin{equation}\label{tildex_in_caut_radius}
    \mathcal{O} \cap\pi\circ\FLO^+(\xi_0) \subset \{x=\gamma_\eta(s): \; \eta\in L_{x_0}^{*,+}M \text{ and } s< \rho(\eta)\}.
\end{equation}

We begin with the following auxiliary lemmas.

\begin{lemma}\label{back flowout misses U}
Let $U\subset T^*M$ be an open set containing $\xi_0$ and $-\xi_0$. Then, for sufficiently small $a>0$,
\begin{equation}
    \FLO^-(\tilde \xi) \cap T^*B_G(x_0;a) \subset U
\end{equation}\label{back flowout misses U 0}
for all $\tilde \xi \in L^*_{\tilde x} M$.
\end{lemma}
\begin{proof}
    Assume that \eqref{back flowout misses U 0} does not hold. Then there exist a sequence $a_k$ tending to zero, unit covectors $\tilde\xi(a_k)\in L_{\tilde x}^{*,-}M$, and $\eta(a_k)\in L^{*,-}M$ such that 
    \begin{equation}\label{back flowout misses U 1}
        \eta(a_k)\in \FLO^-(\tilde \xi(a_k)) \cap T^*B_G(x_0;a_k) \qquad \text{and} \qquad \eta(a_k) \notin U.
    \end{equation}
    Let $\tilde \xi\in L_{\tilde x}^{*,-}M$ be a limit point of the sequence $\tilde\xi(a_k)$. Without loss of generality, 
    $\tilde\xi(a_k) \rightarrow \tilde \xi$ as $k\rightarrow \infty$. Then, $x_0\in \pi\circ \FLO^-(\tilde \xi)$ and 
    \begin{equation}
        \eta(a_k) \rightarrow \eta := \FLO^-(\tilde\xi ) \cap T_{x_0}^*B_G(x_0;a)\qquad \text{as } k\rightarrow \infty.
    \end{equation}
    Additionally, we know that $\tilde x\in \pi\circ \FLO^+(\xi_0)$, and hence, there exists $\tilde\xi_0 \in L_{\tilde x}^{*,-}M$ such that $x_0 \in \exp_{\tilde x}(\tilde \xi_0)$. Without loss of generality, $\Phi_1(\tilde\xi_0) = \xi_0$. Since \eqref{tildex_in_caut_radius}, Lemmas 6.8 and 6.5 in \cite{feizmohammadi2020inverse} imply that $\eta = \alpha \xi_0$ for some non-zero $\alpha$.
    
    Then, by hypothesis, $\eta\in U$. This contradicts to \eqref{back flowout misses U 1} and the fact that $U$ is an open conic set, and hence, \eqref{back flowout misses U} holds. 
\end{proof}

\begin{proposition}
\label{wf of u123}
If $(\xi_0,\xi_1,\xi_2,\xi_3)\in R$, then $\xi_0 \in \WF(u_{123})$ or $-\xi_0\in \WF(u_{123})$ for all $h>0$ sufficiently small.
\end{proposition}
\begin{proof}
    By definition $R$, we know that there is $\tilde x \in {\mathcal O} \cap \FLO^+(\xi_0)$ such that 
    \begin{equation*}
        \tilde x \neq \pi(\xi_0), \qquad \tilde x \notin \mu([0,1]), \qquad \hat{\tilde x} \notin \bigcup_{j=1}^3 \pi\circ \FLO^+(\xi_j),
    \end{equation*}
    and
    \begin{equation}\label{x_hat_singsupp}
        \hat{x}(\tilde x) \in \singsupp(u_{0123456}).
    \end{equation}
    Moreover, conditions \eqref{distinction} and \eqref{hat x tilde x} are satisfied. We choose $h>0$ small enough so that for any $a\in (0,h)$ the hypotheses of Lemma \ref{back flowout misses U} and all results of the previous sections are fulfilled. Decomposition \eqref{decomp_u0123456} and Lemma \ref{ureg_is_smooth} imply
    \begin{equation*}
        \hat{x}(\tilde x) \in \singsupp(u_{\rm sing}),
    \end{equation*}
    for any $a\in (0,h)$, where 
    \begin{equation*}
        \Box_g u_{\rm sing} = C u_5u_6u_{01234}.
    \end{equation*}
    Since \eqref{x_hat_singsupp} holds for all $a\in (0,h)$, we conclude that 
    \begin{equation*}
        \tilde x \in \singsupp(u_{01234}).
    \end{equation*}
    Similarly, by decomposition \eqref{decomp_u01234} and Lemma \ref{vreg is smooth},
    \begin{equation}
        \tilde x \in \singsupp(v_{\rm sing}), \qquad \text{for all } a\in (0,h),
    \end{equation}
    where 
    \begin{equation*}
        \Box_g v_{\rm sing} = C u_0u_4u_{123}.
    \end{equation*}
    Since $\tilde x\notin \supp{u_0u_4} \subset B_G(x_0;a)$ for sufficiently small $a\in(0,h)$, we derive that $\WF_{\tilde x} (v_{\rm sing}) \subset L_{\tilde x}^*M$. Moreover, the last equation implies that 
    \begin{equation*}
        \FLO^-\left(\WF_{\tilde x} (v_{\rm sing})\right) \cap \WF(u_0u_4u_{123}) \neq \emptyset.
    \end{equation*}
    Therefore, since $\WF(u_0u_4u_{123}) \subset T^* B_G(x_0,a)\cap \WF(u_{123})$, there exists an element
    \begin{equation*}
        \eta \in \FLO^-\left(\WF_{\tilde x} (v_{\rm sing})\right) \cap T^* B_G(x_0,a)\cap \WF(u_{123}).
    \end{equation*}
    By Lemma \ref{back flowout misses U}, any open conic neighbourhood of $\{-\xi_0, \xi\}$ contains $\eta$. In particular any open conic neighbourhood of $\{-\xi_0, \xi\}$ intersects $\WF(u_{123})$. Recalling that $\WF(u_{123})$ is closed, we complete the proof.
\end{proof}

Now we are ready to prove the main result of this section.
\begin{proposition}
\label{necessary condition}
If $(\xi_0,\xi_1,\xi_2,\xi_3)\in R$ then 
$$\pi\circ\FLO^-(\xi_0) \cap \bigcap_{j=1}^3 \pi\circ \FLO^+(\xi_j) \neq\emptyset.$$
\end{proposition}
\begin{proof}
Due to Proposition \ref{wf of u123}, there is $h_0>0$ such that $\xi_0\in \WF(u_{123})$ or $-\xi_0\in \WF(u_{123})$ for all $h\in (0,h_0)$. Let us assume $\xi_0\in \WF(u_{123})$ as the argument for the other case is analogous. The distribution $u_{123}$ solves
\begin{equation*}
    \begin{cases}
        \Box_g u_{123} = -6 u_1 u_2 u_3,\\
        u_{123}\arrowvert_{(-\infty, 0)\times M_0} = 0.
    \end{cases}
\end{equation*}
This means that 
$$\FLO^-(\xi_0) \cap \WF(u_1u_2u_3)\neq \emptyset$$
for all $h\in (0,h_0)$. We also know that
$$\WF(u_1u_2u_3) \subset \left(\bigcup_{j=1}^3 \WF(u_j)\right) \cup \left(\bigcup_{j,k=1}^3(\WF(u_j)+ \WF(u_k))\right)\cup \left(\sum_{j=1}^3\WF(u_j)\right).$$
Since $\FLO^-(\xi_0)\subset L^*M$ and 
$$(\WF(u_j)+ \WF(u_k))\cap L^*M\subset \WF(u_j) \cup \WF(u_k),$$
it follows
$$\FLO^-(\xi_0)\cap \WF(u_1u_2u_3) \subset\left( \FLO^-(\xi_0)\cap \bigcup_{j=1}^3 \WF(u_j)\right) \cup\left( \FLO^-(\xi_0)\cap\sum_{j=1}^3\WF(u_j)\right),$$
for all $h\in (0,h_0)$. In light of Lemma \ref{uj smooth}, we know that $$\left( \FLO^-(\xi_0)\cap \bigcup_{j=1}^3 \WF(u_j)\right) = \emptyset,$$
for all $h\in(0,h_0)$. It must therefore hold that 
$$\left( \FLO^-(\xi_0)\cap\sum_{j=1}^3\WF(u_j)\right) \neq \emptyset,$$ 
and hence,
\begin{eqnarray}
\label{FLO- intersects singsupp}
\pi\circ\FLO^-(\xi_0) \cap \bigcap_{j=1}^3 \singsupp(u_j)\neq \emptyset
\end{eqnarray}
for all $h\in (0,h_0)$. Now each $u_j$ solves
\begin{equation*}
    \begin{cases}
        \Box_g u_j = f_j,\\
        u_j\arrowvert_{(-\infty, 0)\times M_0} = 0
    \end{cases}
\end{equation*}
with $\WF(f_j) \subset \left( \B_h(\xi_j) \cup \B_h(-\xi_j)\right)\cap T^*_{x_j}M$. So for each $h\in (0,h_0)$,
$$\singsupp(u_j)\subset \pi\circ \FLO^+(\B_h(\xi_j)\cap L^*_{x_j}M).$$
Substitute this into \eqref{FLO- intersects singsupp} we have that 
$$\pi\circ\FLO^-(\xi_0) \cap \bigcap_{j=1}^3 \pi\circ \FLO^+(\B_h(\xi_j)\cap L^*_{x_j}M)\neq \emptyset$$
for all $h>0$. Taking intersection overall $h\in (0,h_0)$ and use compactness of the causal diamond due to global hyperbolicity we have that
$$\pi\circ\FLO^-(\xi_0) \cap \bigcap_{j=1}^3 \pi\circ \FLO^+(\xi_j)\neq \emptyset.$$
\end{proof}

Finally, we show that this result also holds for elements of $\overline{R}$:

\begin{corollary}
If $(\xi_0,\xi_1,\xi_2,\xi_3) \in \overline R$ then 
$$\pi\circ\FLO^-(\xi_0) \cap \bigcap_{j=1}^3 \pi\circ \FLO^+(\xi_j) \neq\emptyset.$$
\end{corollary}
\begin{proof}
Let $(\xi_0^j,\xi_1^j, \xi_2^j, \xi_3^j) \in R$ with 
$$(\xi_0^j,\xi_1^j, \xi_2^j, \xi_3^j) \to (\xi_0,\xi_1, \xi_2, \xi_3).$$
By Proposition \ref{necessary condition} we have for each $j\in\nn$ there exists 
$$y_j\in \pi\circ\FLO^-(\xi^j_0) \cap \bigcap_{k=1}^3 \pi\circ \FLO^+(\xi_k^j).$$

As the sequence $y_j$ is contained in a fixed causal diamond, we may assume the entire sequence converges to $y$ due to global hyperbolicity. By continuity, we then have that 
$$y\in  \pi\circ\FLO^-(\xi_0) \cap \bigcap_{k=1}^3 \pi\circ \FLO^+(\xi_k).$$

\end{proof}

\section{Sufficient condition}

\begin{lemma}\label{suf_cond_demo}
Assume that $(\xi_0, \xi_1, \xi_2, \xi_3) \in \left(L^{*,+} \Omega\right)^4$
satisfies $(a)$, $(b)$, and $(c)$. We also assume that \eqref{no intersect mu} holds for $\xi_0$. Then $(\xi_0, \xi_1, \xi_2, \xi_3)$ contains in $\overline R$.
\end{lemma} 
\begin{proof}
We begin by showing that the second half of the non-return condition holds, that is, \eqref{distinction} is satisfied. Due to conditions $(a)$ and $(b)$, we know that $\gamma_{\xi_0}$ and $\gamma_{\xi_j}$ are distinct and 
\begin{equation*}
    \gamma_{\xi_0}(s_0) = \gamma_{\xi_j}(s_j) = y, \qquad j=1,2,3,
\end{equation*}
for some $y\in M$, $s_0\in (-\rho(\xi_0),0)$, and $s_j\in (0, \rho(\xi_j))$. In particular, $\gamma_{\xi_0}$ is the unique optimizing geodesic segment from $x_0$ to $\gamma_{\xi_0}(s)$ for any $s\in (-\rho(\xi_0),s_0]$. Therefore, $\gamma_{\xi_j}$ does not pass $x_0$, and hence, relation \eqref{distinction} holds.



Next, we will show that the desirable condition holds. For $h_0>0$, we write 
\begin{equation*}
    \mathcal K_{j} := \FLO^+\left((\B_{h_0}(\xi_j)\cup \B_{h_0}(-\xi_j))\cap L_{x_j}^{*,+} M\right), \qquad \text{for } j=1,2,3,
\end{equation*}
where $x_j = \pi(\xi_j)$ and $\B_h(\xi)$ is the set defined by \eqref{def_B}.
As $s_j < \rho(\xi_j)$ by condition $(b)$,
there exists a characteristic submanifold $K_{j} \subset M$ of codimension one such that $\mathcal K_{j} = N^* K_{j}$ near $y$. Since the covectros $\eta_1$, $\eta_2$, and $\eta_3$ are lightlike, they are linearly dependent only if two of them are proportional, which is not the case due to condition $(a)$. Therefore, the covectors $\eta_1$, $\eta_2$, and $\eta_3$ are linearly independent. It follows that $K_{1}$, $K_{2}$, and $K_{3}$ are transversal for small $h_0>0$, and hence, 
$$K := K_{1} \cap K_{2} \cap K_{3}$$ 
is a smooth curve. Moreover, $K$ is a spacelike curve. Indeed, let $\tilde y \in K$ and let $v \in T_{\tilde y} K$ be nonzero. Then, for $j=1,2,3$, it follows that $v \in T_{\tilde y} K_{j}$ and there is $\tilde \eta_j \in L_{\tilde y}^* M$, small perturbations of $\eta_j$, such that $\pair{\tilde \eta_j, v} = 0$. This implies that $v$ is spacelike, and hence, $K$ is a spacelike curve. 

We know that
    \begin{align*}
N_{y}^* K 
= 
N_{y}^* K_1 \oplus N_{y}^* K_2 \oplus N_{y}^* K_3
= \spn(\eta_1, \eta_2, \eta_3),
    \end{align*}
and hence, condition $(c)$ implies $\eta_0 \in N_{y}^* K$. Therefore, due to Proposition~\ref{flowout is conormal}, there exist sequences
\begin{equation*}
     \eta_{0,k} \in N^* K\cap L^{+,*}M, \qquad s_{0,k}>0,
\end{equation*}
converging to $\eta_0$ and $s_0$, respectively, a sequence of sufficiently small $h_k>0$, and open sets $\mathcal{O}_k$ containing $x_{0,k} := \gamma_{\eta_{0,k}}(s_{0,k})$ such that
\begin{align}\label{def_Yk}
    \FLO^+(N^* \Gamma_{k;h_k} \cap L^* M) = N^* Y_k
\end{align}
in $\mathcal O_k$ for some characteristic submanifolds $Y_k \subset M$ of codimension one. Here, 
$$
\Gamma_{k;h_k} = \{y\in K\mid d_G(y,y_k)<h_k\},
$$
and $y_k = \pi(\eta_{0,k}) \in K$.
As $\eta_{0,k} \in N_{y_k}^* K \cap L^{*,+}M$ there exist sequences
\begin{equation*}
    \eta_{j,k} \in N_{y_k}^* K_j \cap L^{*,+}M, \qquad j=1,2,3,
\end{equation*}
such that $\eta_{0,k} \in \spn(\eta_{1,k}, \eta_{3,k}, \eta_{3,k})$.
Moreover, since $K_j$ are submanifolds of codimension one and $\eta_{0,k} \to \eta_0$, we may choose $\eta_{j,k}$ so that $\eta_{j,k} \to \eta_j$.

Let $j=1,2,3$ and $k\in \mathbb{N}$, then we denote by $\xi_{j,k}$ the covector version of $\dot \gamma_{\eta_{j,k}}(-s_{j,k})$, where $s_{j,k} > 0$ is chosen so that $\pi(\xi_{j,k}) = x_j$.
Since \eqref{no intersect mu} is satisfied, it follows that  $\gamma_{\xi_{0,k}}$ does not intersect $\hat x(x_{0,k})$ for 
sufficiently large $k$. Moreover, for any $\tau_0 > 0$ small enough, the map
\begin{eqnarray} 
\label{rmapsto}
(0,\tau_0)\ni \tau\mapsto \hat x(\gamma_{\xi_{0,k}}(\tau))
\end{eqnarray}
is injective. Otherwise, a pair of points on a lightlike geodesic segment contained in $\Omega$ would be joined by two distinct causal geodesics which contradicts our assumption that lightlike segments in $\Omega$ do not contain a pair of conjugate points. We can conclude then that the range of the map \eqref{rmapsto} must contain infinitely many points. Therefore we can find $\tau_k \in (0,r_0)$ such that $\tilde x_k := \gamma_{\xi_{0,k}}(\tau_k)$ satisfies
    \begin{align*}
\hat x(\tilde x_k) \notin \bigcup_{j=1}^3 \left(\pi\circ\FLO^+(\xi_{j,k})\right).
    \end{align*}
Moreover, by choosing $\tau_0$ small enough, we are able to make $\tilde x_k$ to be as close as we want to $x_{0,k}$. Therefore, condition \eqref{hat x tilde x} is satisfied.

Next, we aim to show that $(\xi_{0,k}, \xi_{1,k}, \xi_{2,k}, \xi_{3,k})$ satisfies the remaining part of the desirable condition for sufficiently large $k$. To do this, we choose $\{f_{j,k}\}_{j=0}^6$ to be the distributions defined by \eqref{fj}, \eqref{f0f4}, \eqref{f5}, and \eqref{f6}, where 
$$\left\{\xi_{1}; \xi_{2}; \xi_{3}; x_{0}; \tilde x\right\}$$ 
is replaced by 
$$\left\{\xi_{1,k}; \xi_{2,k}; \xi_{3,k}; x_{0,k}; \tilde x_k\right\},$$ respectively. Correspondingly, we define the linearized solutions as in \eqref{linearized_sol}:
\begin{equation*}
    u_{m}^k, \quad u_{m,n}^k,\quad \cdots \quad, u_{0123456}^k, \qquad \text{for } m,n \in \{0,1,\cdots,6\} \text{ and } k\in \mathbb N,
\end{equation*}
where the sources $\{f_{j}\}_{j=0}^6$ are replaced by $\{f_{j,k}\}_{j=0}^6$.

By definition, $u_{j}^k$ is singular on 
    \begin{align*}
\mathcal K_{j,k;h} = \FLO^+((\B_{h}(\xi_{j,k})\cup \B_{h}(-\xi_{j,k}))\cap L_{x_{j}}^{*,+} M), \qquad \text{for } j=1,2,3 \text{ and } h>0.
    \end{align*}
If $h>0$ is small and $k$ is large enough, then $\mathcal K_{j,k;h} \subset \mathcal K_j$. Moreover, near $y_k= \pi(\eta_{0,k})$ we write $\mathcal K_{j,k;h} = N^* K_{j,k;h}$, and hence, $\bigcap_{j=1}^3 K_{j,k;h} \subset \Gamma_{k;h_k}$ for sufficiently small $h > 0$.

Furthermore, near $y_k$, $u_{j}^k$ is a conormal distribution associated to $N^* K_{j,k;h}$ and 
\begin{equation}\label{sigmaj}
    \sigma[u_{j}^k](y_k, \pm \eta_{j,k}) \ne 0, \quad \text{for } j=1,2,3.
\end{equation} 
For large $k$, the product $u_{1} u_{2} u_{3}$ is a conormal distribution associated to $N^* \Gamma_{k;h_k}$ microlocally near $\eta_{0,k}$.
This follows from the product calculus of conormal distributions. Indeed, 
analogously to $\eta_3 \notin \spn(\eta_1, \eta_2)$ that was proven above, we have 
$\eta_0 \notin \spn(\eta_{j}, \eta_{j'})$ for all $j,j'=1,2,3$, 
and this implies $\eta_{0,k} \notin \spn(\eta_{j,k}, \eta_{j',k})$ for large $k$.
Due to \eqref{sigmaj}, the product calculus yields that
\begin{equation*}
    \sigma[u_{1}^k u_{2}^k u_{3}^k](y_k, \eta_{0,k}) \ne 0.
\end{equation*}

Let $\beta_{\eta_{0,k}}$ be the bicharacteristic through $\eta_{0,k}$
and write $\xi_{0,k} = \beta_{\eta_{0,k}}(s_{0,k})$.
We note that $\xi_{0,k} \to \xi_0$. Due to the transport equation satisfied by $\sigma[u_{123}^k]$ along $\beta_{\eta_{0,k}}$ we have 
\begin{equation}\label{sigma123}
    \sigma[u_{123}^k](\xi_{0,k}) \ne 0.
\end{equation}
Observe that we need to use the general theory of Lagrangian distributions since there might be focal points of $\pi\circ\FLO^+(N^*\Gamma_{k;h_k}\cap L^*M)$ along $\gamma_{\eta_{0,k}}$ between $y_k$ and $x_{0,k}$.

The second half of the non-return condition, relation \eqref{distinction}, implies that
\begin{equation*}
    \xi_{0,k} \notin \FLO^+(\xi_{j,k}), \qquad j=1,2,3,
\end{equation*}
for sufficiently large $k$. Since conditions \eqref{distinction} and \eqref{hat x tilde x} are satisfied, by Lemma \ref{vreg is smooth}, there are exist a constant $c_4 \ne 0$ and a distribution $r_4$ such that $r_4$ is smooth in the neighbourhood $\mathcal O_k$ chosen in \eqref{def_Yk}
and 
    \begin{align}\label{u01234_eq_u0u4u123}
\Box_g( u_{01234}^k + r_4) = c_4 u_{0}^k u_{4}^k u_{123}^k.
    \end{align}
Observe that $u_{01234}^k$ is a conormal distribution associated to $Y_k$, see (\ref{def_Yk}), near $x_{0,k}$.
Moreover, for each $\xi_{0,k}$ fixed, we can choose $\tau > 0$ so that $\gamma_{\xi_{k,0}}(\tau) \in \mathcal O_k$. We now choose $a > 0$ defined in (\ref{f0f4}) to be sufficiently small so that $\gamma_{\xi_{0,k}}(\tau)\notin\supp(u_0) = \supp(u_4)$. Let us show that
\begin{equation}\label{sigma01234}
    \sigma[u_{01234}^k](\beta_{\xi_{0,k}}(\tau)) \ne 0.
\end{equation}
Indeed, recall that $\Box_g$ is a second-order PDO with a real homogeneous principal symbol. Therefore, it has parametrix $Q\in I^{-\frac{3}{2}, -\frac{1}{2}} (\Delta_{T^*M}, \Lambda_{\Box_g})$, see for instance \cite{Melrose:1979aa}. Moreover, since $$\FLO^+\left( N^*\Gamma_{k,h_k}\cap L^*M\right) \subset \Sigma(\Box_g),$$
Proposition 2.3 in \cite{Greenleaf:1993aa} implies that 
\begin{equation*}
	Q:I^r(N^*Y_k, \Omega^{1/2}) \mapsto I^{r - 1}(N^*Y_k, \Omega^{1/2}).
\end{equation*}
In particular, we know that $u_{01234}^k \in I(N^*Y_k, \Omega^{1/2})$.

Since \eqref{def_Yk} holds in a neighbourhood $\mathcal{O}_k$ of $x_{0,k}$, it follows from Theorem 3.3.4 in \cite{Hormander:1971aa} that the Maslov bundle is trivial near $x_{0,k}$. Hence, we fix a coordinate system near $x_{0,k}$, so that a Maslov factor is just a constant, and hence, it is not involved in differentiation. Therefore, we define the Lie derivative by differentiating the pullback of the flow corresponding to the Hamiltonian $H$ associated with $\sigma[\Box_g]$. We denote the Lie derivative by $\mathcal{L}_H$. According to Theorem 5.3.1 in \cite{Duistermaat_Hormander:1972aa}, the following identity holds 
\begin{equation*}
	\imath^{-1} \mathcal{L}_H\sigma[u_{01234}^k] = \sigma[\Box_g u_{01234}^k] = \sigma[u_0^ku_4^ku_{123}^k].
\end{equation*}

Let $\omega = |g|^{1/4}$, then we can express $ \sigma[u_{01234}] = \alpha\omega$ for some smooth function $\alpha$. We set
\begin{equation*}
	\phi_k(s) := \alpha \circ\beta_{\xi_{0,k}} (s),
	\qquad 
	\psi_k(s) := \int_{0}^{s}( \mathrm{div}_\omega H) \circ \beta_{\xi_{0,k}} (\tau) d\tau.
\end{equation*}
Then, we have 
\begin{equation*}
	\imath\omega^{-1}\sigma[\Box_g u_{01234}^k] \circ \beta_{\xi_{0,k}} = \omega^{-1}\mathcal{L}_H (\alpha\omega) \circ \beta_{\xi_{0,k}} = e^{-\psi_k}\partial_s(e^{\psi_k}\phi_k),
\end{equation*}
or equivalently, 
\begin{equation*}
	\imath\sigma[\Box_g u_{01234}^k] = \nabla_s^{\omega, k} \sigma[u_{01234}^k], \text{where }  \nabla_s^{\omega, k}:= e^{-\psi_k} \omega \circ \partial_s \circ e^{\psi_k} \omega^{-1}.
\end{equation*}
Therefore, taking into account the supports of $u_0^k$ and $u_4^k$, by integration, we derive 
\begin{equation*}
	\sigma[u_{01234}^k](\beta_{\xi_{0,k}}(s)) = \imath e^{ - \psi_k(s)} \omega(\beta_{\xi_{0,k}}(s)) \int_{0}^{s}  e^{\psi_k(\tau)} \omega^{-1}(\beta_{\xi_{0,k}}(\tau))  \sigma[u_0^ku_4^ku_{123}^k] (\beta_{\xi_{0,k}})(\tau) d\tau ,
\end{equation*}
for $s>0$ such that $\beta_{\xi_{0,k}}(s) \in \mathcal{O}_k\setminus \mathrm{supp}(u_0^k)$. Since \eqref{sigma123} and $u_0^k$, $u_4^k$ are smooth functions supported in $B_G(x_{0,k};a)$, we conclude that the right-hand side of the last equation is non-zero for sufficiently small $a>0$. Therefore, \eqref{sigma01234} holds.

By Lemma \ref{ureg_is_smooth}, there exist a constant $c_6^k \ne 0$ and a distribution $r_6^k$ such that $\mu^* r_6^k$ is smooth at $\hat x(\tilde x_k)$ and
\begin{equation*}
    u_{0123456}^k = v_6^k + r_6^k,
\end{equation*}
and
\begin{equation*}
    \Box_g v_6^k = c_6 u_5^k u_6^k u_{01234}^k.
\end{equation*}
Recall that we choose $\tau_k$ such that $\tilde x_k = \gamma_{\xi_{0,k}}(\tau_k)$ satisfies \eqref{hat x tilde x}. Let us set $\tilde \xi_k = \beta_{\xi_{0,k}}(\tau_k)$ and write $\tilde x_k = (\tilde t_k, \tilde x_k')$. Since $Y_k$ and $\T_{\tilde x_k'}$ are transversal, microlocally near any $\xi \in L^{*,+} M \setminus \tilde \xi_k$, we know that $u_5^k u_6^k u_{01234}^k$ is a conormal distribution associated to the conormal bundle of the point $\{\tilde x_k\}=Y_k \cap \T_{\tilde x_k'}$. By Lemma \ref{lem_ximu}, we know that
\begin{equation*}
    \xi_k^\mu = \tilde{\xi}_k + \nu(\tilde{\xi}_k)
\end{equation*}
where $\nu(\tilde{\xi}_k) \in N^*_{\tilde{x}_k}\T_{\tilde{x}_k'}/\R^+$ is a covector defined by \eqref{def_nu}. Moreover, by \eqref{sigma01234} and definition of $u_6^k$,
\begin{equation*}
    \sigma[u_{01234}^k] (\tilde{\xi}_k) \neq 0,
    \qquad
    \sigma[u_6^k] (\nu(\tilde{\xi}_k)) \neq 0.
\end{equation*}
Therefore, the product calculus of conormal distributions implies that 
$$\sigma[u_5^k u_6^k u_{01234}^k](\xi^\mu_k) \ne 0.$$
Due to the transport equation satisfied by $\sigma[v_6^k]$ along $\beta_{\xi_k^\mu}$, we know that
$$\sigma[v_6^k](\beta_{\xi_k^\mu}(s)) \ne 0,$$
where $s > 0$ satisfies $\gamma_{\xi_k^\mu}(s) = \hat x(\tilde x)$.
Using Lemma \ref{lemma_2.12}, we see that $\mu^* u_{0123456}^k$ is singular at $\mu^{-1}(\hat x(\tilde x_k))$.
This in addition with \eqref{hat x tilde x} shows that $(\xi_{0,k}, \xi_{1,k}, \xi_{2,k}, \xi_{3,k}) \in R$ for large $k$. It then follows that $(\xi_0, \xi_1, \xi_2, \xi_3) \in \overline R$.
\end{proof} 

Next, we want to improve the last theorem, namely, we aim to remove assumption \eqref{no intersect mu}. To do this, we will need the following lemmas:
\begin{lemma}\label{lem_suff_pert}
    Suppose that $(\xi_0, \xi_1, \xi_2, \xi_3) \in (L^{*,+} \Omega)^4$ satisfies $(a)$, $(b)$, and $(c)$. Assume that a sequence $\xi_{0,k} \in L^{*,+} \Omega$ converges to $\xi_0$. Then there exists a sequence
    \begin{equation*}
        (\xi_{0,k_l}, \xi_{1,l}, \xi_{2,l}, \xi_{3,l}) \in (L^{*,+} \Omega)^4
    \end{equation*}
    satisfying $(a)$, $(b)$, and $(c)$ such that $\xi_{0,k_l}$ is a subsequence of $\xi_{0,k}$ and $\xi_{j,l} \rightarrow \xi_j$ for $j=1,2,3$.
\end{lemma}

The proof of this result is based on the following elementary lemma. 

\begin{lemma}\label{lem_linalg}
Let $L \subset \R^{1+3}$ be the light cone with respect to the Minkowski metric. 
Let $\xi_1, \xi_2, \xi_3 \in L$ be linearly independent.
Suppose that $\xi_0 \in \spn(\xi_1, \xi_2, \xi_3)$
and that $\xi_0 \notin \spn(\xi_2, \xi_3)$.
Let $U \subset L$ be a neighbourhood of $\xi_1$.
Then there is a neighbourhood $V \subset \R^{1+3}$ of $\xi_0$ such that for all $\tilde \xi_0 \in V$ there is $\tilde \xi_1 \in U$ such that $\tilde \xi_0 \in \spn(\tilde \xi_1, \xi_2, \xi_3)$.
\end{lemma} 
\begin{proof}
The statement is invariant with respect to a non-vanishing rescaling of $\xi_j$, $j=1,2,3$, and we may assume without loss of generality that $\xi_j = (1, \xi_j')$ with $\xi_j'$ a unit vector in $\R^3$. After a rotation in $\R^3$, we may assume that 
    \begin{align*}
\xi_1' = (1,0,0), \quad \xi_2' = (a,b,0), \quad \xi_3'=(c,d,e),
    \end{align*}
for some $a,b,c,d,e \in \R$. We write $e_j$, $j=0,1,2,3$, for the usual orthonormal basis of $\R^{1+3}$.

To get a contradiction we suppose that neither $B_1$ nor $B_2$ is a basis of $\R^{1+3}$, where
    \begin{align*}
B_1 = (e_3, \xi_1, \xi_2, \xi_3),
\quad B_2 = (e_2, \xi_1, \xi_2, \xi_3).
    \end{align*}
As $\xi_1, \xi_2, \xi_3$ are linearly independent but 
$B_1$ is not a basis, there holds $e_3 \in \spn(\xi_1, \xi_2, \xi_3)$. In particular, $e \ne 0$. 
Furthermore, as $B_2$ is not a basis, also $e_2 \in \spn(\xi_1, \xi_2, \xi_3)$. As $e \ne 0$, there are, in fact, $a_1, a_2 \in \R$ such that 
$e_2 = a_1\xi_1 + a_2 \xi_2$. In particular, $a_2 b = 1$ and $a_2, b \ne 0$. Hence
    \begin{align*}
a_1/a_2 (1,1) + (1,a) = 0,
    \end{align*}
and this again implies $a = 1$. But $a = 1$ and $b \ne 0$ is a contradiction with $\xi_2' = (a,b,0)$ being a unit vector. 

We define
    \begin{align*}
\tilde \xi_1(\theta) = 
\begin{cases}
(1,\cos(\theta), 0, \sin(\theta)), & \text{if $B_1$ is a basis,}
\\
(1,\cos(\theta), \sin(\theta), 0), & \text{otherwise.}
\end{cases} 
    \end{align*}
Note that $\tilde \xi_1(0) = \xi_1$ and that
$(\p_\theta \tilde \xi_1(0), \xi_1, \xi_2, \xi_3)$ is a basis of $\R^{1+3}$.
We conclude by applying the implicit function theorem to 
    \begin{align*}
F(\tilde \xi_0, \theta, c_1, c_2, c_3) = c_1 \tilde \xi_1(\theta) + c_2 \xi_2 + c_3 \xi_3 - \tilde \xi_0.
    \end{align*}
\end{proof} 

Now we are ready to proof Lemma \ref{lem_suff_pert}.

\begin{proof}[Proof of Lemma \ref{lem_suff_pert}]
Set $y_k = \gamma_{\xi_{0,k}}(s_0)$
and denote by $\alpha_k$ the shortest path from $y$ to $y_k$ with respect to Riemannian metric $G$. For $j=0,1,2,3$, define $\eta_{j,k}$ as the parallel transport of $\eta_j$ from $y$ to $y_k$ along $\alpha_k$, where $\eta_j$ is as in $(c)$. 
Note that $\eta_{j,k} \to \eta_j$ as $k \to \infty$.
Write $\tilde \eta_{0,k}$ for the covector version of 
$\dot \gamma_{\xi_{0,k}}(s_0)$.
Also $\tilde \eta_{0,k} \to \eta_0$ as $k \to \infty$.

Choose an orthornormal basis at $y$ and denote by $g'_k$, $\eta_{j,k}'$ and $\tilde \eta_{0,k}'$
the representations of $g$, $\eta_{j,k}$ and $\tilde \eta_{0,k}$ in the basis at $y_k$, obtained as the parallel transport of the basis at $y$ along $\alpha_k$.
Now $g'_k$ is the Minkowski metric and $\eta_{j,k}'$ is independent of $k$. We write $\eta_{j,k}' = \eta_j'$ and have $\tilde \eta_{0,k}' \to \eta_0'$. Observe that $\eta_0' \notin \spn(\eta_2', \eta_3')$ due to $(a)$, and $\eta_0' \in \spn(\eta_1', \eta_2', \eta_3')$ due to $(c)$.

Let $U_l \subset \R^{1+3}$ be the intersection of $L$ with the Euclidean ball of radius $1/l$ centered at $\eta_1'$.
By Lemma \ref{lem_linalg} there is a neighbourhood $V_l \subset \R^{1+3}$ of $\eta_0'$ such that for all $\tilde \eta_0' \in V_l$ there is $\tilde \eta_1' \in U_l$ such that $\tilde \eta_0' \in \spn(\tilde \eta_1', \eta_2', \eta_3')$.

Let $k_l$ be large enough so that $\tilde \eta_{0,k_l}' \in V_l$.
We denote by $\xi_{j,l}$ the covector version of $\dot \gamma_{\eta_{j,k_l}}(-s_j)$ for $j=2,3$
and let $\xi_{1,l}$ be the covector version of $\dot \gamma_{\tilde \eta_{1,l}}(-s_1)$
where $\tilde \eta_{1,l} \in L_{y_{k_l}}^* M$ is the covector corresponding to a choice of $\tilde \eta_{1,l}' \in U_l$ such that $\tilde \eta_{0,k_l}' \in \spn(\tilde \eta_{1,l}', \eta_2', \eta_3')$.
Now $\xi_{j,l} \to \xi_j$ as $l \to \infty$ for each $j=1,2,3$.
Moreover, 
$$(\xi_{0,k_l}, \xi_{1,l}, \xi_{2,l}, \xi_{3,l})\in (L^{*,+} \Omega)^4$$  satisfies $(a)$ and $(b)$ for large $l$ due to this convergence. It also satisfies $(c)$ by the above construction.
\end{proof}

Finally, we show that condition \eqref{no intersect mu} can be removed from Lemma \ref{suf_cond_demo}:

\begin{lemma}
Assume that $(\xi_0, \xi_1, \xi_2, \xi_3) \in \left(L^{*,+} \Omega\right)^4$
satisfies $(a)$, $(b)$, and $(c)$. Then $(\xi_0, \xi_1, \xi_2, \xi_3)$ contains in $\overline R$.
\end{lemma} 
\begin{proof} 
    We choose such $\xi_{0,k} \in L^{*,+} \Omega$, converging to $\xi_0$,
    that (\ref{no intersect mu}) holds with $\xi_0$ replaced by $\xi_{0,k}$.
    By Lemma \ref{lem_suff_pert}, for each $j=1,2,3$ there is such a sequence $\xi_{j,l} \in L^{*,+} \Omega$ converging to $\xi_j$ that 
    $$(\xi_{0,k_l}, \xi_{1,l}, \xi_{2,l}, \xi_{3,l})$$ 
    satisfies $(a)$, $(b)$, and $(c)$ for a subsequence $\xi_{0,k_l}$ of $\xi_{0,k}$. 
    Therefore, by Lemma \ref{suf_cond_demo}, 
    $$(\xi_{0,k_l}, \xi_{1,l}, \xi_{2,l}, \xi_{3,l}) \in \overline R$$
    and hence also $(\xi_0, \xi_1, \xi_2, \xi_3) \in \overline R$.
\end{proof}

\section*{Acknowledgement}
We would like to thank Matti Lassas for the useful discusion.

\section*{Data Availability}
All data generated or analyzed during this study are contained in this document.

\bibliography{_biblio}
\bibliographystyle{amsalpha}

\end{document}